\documentclass[notitlepage,11pt,reqno]{amsart}
\usepackage[foot]{amsaddr}
\usepackage[numbers,sort]{natbib}

\usepackage{amssymb,nicefrac,bm,upgreek,mathtools,verbatim,esint}
\usepackage[final]{hyperref}
\usepackage[mathscr]{eucal}
\usepackage{dsfont}

\usepackage{color}
\definecolor{dmagenta}{rgb}{.4,.1,.5}
\definecolor{dblue}{rgb}{.0,.0,.5}
\definecolor{mblue}{rgb}{.0,.0,.8}
\definecolor{ddblue}{rgb}{.0,.0,.4}
\definecolor{dred}{rgb}{.6,.0,.0}
\definecolor{dgreen}{rgb}{.0,.5,.0}
\definecolor{Eeom}{rgb}{.0,.0,.5}

\usepackage[normalem]{ulem}

\usepackage[margin=1in]{geometry}
\allowdisplaybreaks

\newtheorem{lemma}{Lemma}[section]
\newtheorem{theorem}{Theorem}[section]

\newtheorem{corollary}{Corollary}[section]

\theoremstyle{definition}
\newtheorem{definition}{Definition}[section]

\theoremstyle{remark}

\newtheorem{remark}{Remark}[section]

\numberwithin{equation}{section}

\hypersetup{
  colorlinks=true,
  citecolor=dblue,
  linkcolor=dblue,
  frenchlinks=false,
  pdfborder={0 0 0},
  naturalnames=false,
  hypertexnames=false,
  breaklinks}
\usepackage[capitalize,nameinlink]{cleveref}

\crefname{section}{Section}{Sections}
\crefname{subsection}{Subsection}{Subsections}
\crefname{condition}{Condition}{Conditions}
\crefname{hypothesis}{Hypothesis}{Conditions}
\crefname{assumption}{Assumption}{Assumptions}
\crefname{lemma}{Lemma}{Lemmas}
\crefname{claim}{Claim}{Claims}

\Crefname{figure}{Figure}{Figures}

\crefformat{equation}{\textup{#2(#1)#3}}
\crefrangeformat{equation}{\textup{#3(#1)#4--#5(#2)#6}}
\crefmultiformat{equation}{\textup{#2(#1)#3}}{ and \textup{#2(#1)#3}}
{, \textup{#2(#1)#3}}{, and \textup{#2(#1)#3}}
\crefrangemultiformat{equation}{\textup{#3(#1)#4--#5(#2)#6}}%
{ and \textup{#3(#1)#4--#5(#2)#6}}{, \textup{#3(#1)#4--#5(#2)#6}}%
{, and \textup{#3(#1)#4--#5(#2)#6}}

\Crefformat{equation}{#2Equation~\textup{(#1)}#3}
\Crefrangeformat{equation}{Equations~\textup{#3(#1)#4--#5(#2)#6}}
\Crefmultiformat{equation}{Equations~\textup{#2(#1)#3}}{ and \textup{#2(#1)#3}}
{, \textup{#2(#1)#3}}{, and \textup{#2(#1)#3}}
\Crefrangemultiformat{equation}{Equations~\textup{#3(#1)#4--#5(#2)#6}}%
{ and \textup{#3(#1)#4--#5(#2)#6}}{, \textup{#3(#1)#4--#5(#2)#6}}%
{, and \textup{#3(#1)#4--#5(#2)#6}}

\crefdefaultlabelformat{#2\textup{#1}#3}

\makeatletter
\DeclareRobustCommand\widecheck[1]{{\mathpalette\@widecheck{#1}}}
\def\@widecheck#1#2{%
    \setbox\z@\hbox{\m@th$#1#2$}%
    \setbox\tw@\hbox{\m@th$#1%
       \widehat{%
          \vrule\@width\z@\@height\ht\z@
          \vrule\@height\z@\@width\wd\z@}$}%
    \dp\tw@-\ht\z@
    \@tempdima\ht\z@ \advance\@tempdima2\ht\tw@ \divide\@tempdima\thr@@
    \setbox\tw@\hbox{%
       \raise\@tempdima\hbox{\scalebox{1}[-1]{\lower\@tempdima\box
\tw@}}}%
    {\ooalign{\box\tw@ \cr \box\z@}}}
\makeatother

\DeclareMathOperator{\Exp}{\mathbb{E}} 
\DeclareMathOperator{\Prob}{\mathbb{P}} 
\newcommand{\D}{\mathrm{d}}          
\newcommand{\RR}{\mathbb{R}}         
\newcommand{\Rd}{{\mathbb{R}^d}}       
\newcommand{\Ind}{\mathds{1}}            

\newcommand{\Cc}{C}                  

\newcommand{\uuptau}{\Breve{\uptau}}

\newcommand{\eps}{\varepsilon}

\newcommand{\cE}{\mathcal{E}}
\newcommand{\cF}{\mathcal{F}}
\newcommand{\sF}{\mathfrak{F}}    

\newcommand{\sM}{\mathscr{M}}     
\newcommand{\cN}{\mathcal{N}} 
\newcommand{\sR}{\mathscr{R}}     
\newcommand{\sV}{\mathscr{V}}    
\newcommand{\cX}{\mathcal{X}}    

\newcommand{\abs}[1]{\lvert#1\rvert}
\newcommand{\norm}[1]{\lVert#1\rVert}

\DeclareMathOperator*{\esssup}{ess\,sup}
\DeclareMathOperator*{\essinf}{ess\,inf}

\DeclareMathOperator{\cp}{Cap}

\DeclareMathOperator{\ML}{\mathrm{ML}}
\begin{document}

\title[]{Faber-Krahn type inequalities and uniqueness of positive solutions on metric measure spaces}

\author{Anup Biswas}
\address{ Department of Mathematics, Indian Institute of Science Education and Research, Dr. Homi Bhabha Road, Pashan, Pune 411008, India}
\email{anup@iiserpune.ac.in}

\author{Janna Lierl}
\address{Department of Mathematics, University of Connecticut, 341 Mansfield Road, Storrs, CT 06269, USA}
\email{janna.lierl@uconn.edu}

\begin{abstract}
We consider a general class of metric measure spaces equipped with a regular Dirichlet form and then provide a lower bound on the hitting time
probabilities of the associated Hunt process. Using these estimates we establish (i) a generalization of
the classical Lieb's inequality on metric measure spaces and (ii) uniqueness of nonnegative super-solutions on metric measure spaces. Finally, using heat-kernel
estimates we generalize the \textit{local Faber-Krahn} inequality recently obtained in \cite{LS18}.
\end{abstract}



\maketitle

\tableofcontents

\section{Introduction}

In this article we are concerned with three problems, namely (a) generalized Lieb's inequality (b) uniqueness of non-negative super-solutions and (c) local Faber-Krahn estimate, of  
seemingly different flavor but related by heat kernel estimates and  hitting time estimates. The central theme of this article is to showcase how probabilistic method can be applied
to address the above mentioned analytic questions in a very general setting of metric measure spaces.

In an influential work \cite{L83} Lieb showed that given any $\varepsilon\in (0,1)$ there exists $r_\eps$ such that for any domain $D\subset\Rd$, with $\lambda_D$ being its Dirichlet principal eigenvalue in $D$ for the
Laplacian, it holds that
$$\abs{D\cap B(x, r_\eps \lambda_D^{-\nicefrac{1}{2}})}\geq (1-\eps) \abs{B(x, r_\eps \lambda_D^{-\nicefrac{1}{2}})},$$
for some $x\in\Rd$ where $B(x, r)$ denotes the ball of radius $r$ around $x$. The above inequality can be seen as a finer version of the classical Faber-Krahn inequality. This inequality has been
extended in several directions. For instance, \cite{RS18} extends it for Schr\"{o}dinger operators in $\Rd$, \cite{GM16} establishes this on smooth Riemannian manifolds and \cite{B17} obtains an analogous version of this inequality for
fractional Laplacian. In a similar direction we also cite \cite{DCH, DEH} which studies Faber-Krahn type inequalities for the Schr\"{o}dinger operators in $\Rd$ involving singular potentials. In Section \ref{Sec-GL} we show that one
can establish Lieb's inequality in a general setting of metric measure spaces. Our methodology uses the underlying Hunt process and hitting time probabilities. Using similar tools we also establish other
interesting spectral properties such as wavelength density, bounds on negative principal eigenvalue etc.

Our second problem deals with the uniqueness of the non-negative super-solution of 
$$\Delta u + V u^p \leq 0\quad \text{in}\;  K^c,$$
for some compact set $K$. Such problems came to interest due to seminal works of Gidas \cite{G80} and Gidas-Spruck \cite{GS81} which consider the problem in $\Rd, d\geq 3$,  for $V=1$. An enormous amount of
works have been done in generalizing this result in several other situations. \cite{MP99, MP04} use a nonlinear capacity argument together with some integral criterion on the potential $V$ to prove non-existence of non-trivial
supersolutions in $\Rd$. Recently, \cite{GS14} studies the similar problem with $V=1$ on smooth manifolds whereas \cite{GS17} considers this problem in the exterior domain of manifolds. In Section ~\ref{Sec-Uni} we show that
the above hitting time probabilities can be used cleverly to address non-existence of supersolutions in the exterior domain of metric measure spaces.

Our third result is a local Faber-Krahn inequality that is derived from a heat kernel upper bound in a local or non-local regular Dirichlet space. The local Faber-Krahn inequality was introduced in \cite{LS18} for divergence form operators on $\RR^d$ as a refinement of a similar estimate in \cite{RS18}. It states that, if a solution $u$ to the Dirichlet-Schr\"odinger problem in a domain $\Omega$ is large at a point $o \in \Omega$, in the sense that $|u(o)| > \frac{3}{4} \|u\|_{\infty}$, then either $V^-$ is large in some region or the point $o$ is far from the boundary. More specifically, if the process starting from $o$ is likely to reach the boundary within time $T$, then the potential $V^-$ must be large within some ball of radius $\mathcal R(T) = F^{-1}(T)$, that is,
$$\|V^-\|_{L^{\frac{\alpha_1}{\alpha_1 - \alpha_2 + \beta},1}(\Omega \cap B)} \gtrsim 1, $$
where the norm is taken in the appropriate Lorentz space, depending on the volume growth parameters $\alpha_1, \alpha_2$, and the time-space scaling $F(t) \simeq t^{\beta}$.

The rest of the article is organized as follows. Section~\ref{S-prelim} gathers some basic properties of Dirichlet spaces together with heat kernel estimates, whereas
in Section~\ref{S-Hitting} we obtain our main hitting time estimates. Generalized Lieb's inequality is then established in Section~\ref{Sec-GL}. Section~\ref{Sec-Uni}
deals with the uniqueness properties of the positive supersolutions. Finally, in Section~\ref{Sec-FK} we prove the local Faber-Krahn inequality.

\section{Preliminaries and some estimates}\label{S-prelim}
\subsection{Dirichlet forms and heat-kernel estimates}

Let $(\cX,d,\mu,\cE,\cF)$ be a metric measure Dirichlet space. That is, $(\cX,d)$ is a locally compact separable metric space, $\mu$ is a locally finite Radon measure on $X$ with full support, and $(\cE,\cF)$ is a strongly local regular Dirichlet form on $L^2(\cX,d\mu)$. We assume that all metric balls in $(\cX,d)$ are relatively compact. For a domain $\Omega \subset X$, let $C_0(\Omega)$ be the space of continuous function with compact support in $\Omega$. Let $\cF^0(\Omega)$ be the closure of $\cF \cap C_0(\Omega)$ in the norm of $\cF$.


We will consider solutions $u \in \cF^0(\Omega)$ to the Schr\"odinger-Dirichlet problem
\begin{align}  \label{eq:equation}
\begin{split}
 \cE(u,\phi) + \int_{\Omega} Vu \phi \, d\mu &= 0 ,\quad \forall \phi \in \cF^0(\Omega).
\end{split}
\end{align}

The volume of a ball $B(x,r)$ is denoted by $\sV(x,r) = \mu(B(x,r))$.
We assume that the {\em volume doubling property} (VD) holds,
$$ \sV(x, 2r) \lesssim \sV(x, r), \quad r>0, x\in \cX.$$
 The volume doubling property is equivalent to 
\begin{equation}\label{1.2}
\frac{\sV(y, R)}{\sV(x, r)}\lesssim \left(\frac{R + d(x, y)}{r}\right)^{\alpha_2}\quad 0<r\leq R, \ x, y\in \cX,
\end{equation}
see, e.g., \cite{GH09}.
In addition, we will assume the {\em reverse volume doubling property} (RVD),
\begin{align} \label{eq:RVD}
\frac{\sV(x,R)}{\sV(x,r)} \gtrsim \left( \frac{R}{r} \right)^{\alpha_1}, \qquad 0<r\leq R, x \in \cX,
\end{align}
for some $\alpha_1>0$. Note that (RVD) follows from (VD) if the space $\cX$ is connected and unbounded, see \cite[Proposition 3.3]{GHL09}.
The space $(X,d,\mu)$ is called $\alpha$-regular if $\alpha_1 = \alpha_2 = \alpha$.

We also require the mean exit function $F:(0, \infty)\to (0, \infty)$, a $\Cc^1$-function which is strictly increasing and satisfies
\begin{equation}\label{beta}
C^{-1}\left(\frac{R}{r}\right)^{\beta'}\leq\frac{F(R)}{F(r)}\leq C \left(\frac{R}{r}\right)^\beta ,\quad 0<r\leq R\,,
\end{equation}
for some constants $C>0$, $\beta\geq \beta' >1$. Moreover, we assume that
\begin{equation}\label{A1}
0<\inf_{(0, \infty)}\frac{rF'(r)}{F(r)}\leq\sup_{(0, \infty)}\frac{rF'(r)}{F(r)}<\infty.
\end{equation}
We denote by $\mathcal R$ the inverse of $F$.

We set
\[  \|f\|_{\cF} := \left( \cE(f,f) +  \int |f|^2 d\mu \right)^{1/2}, \]
For an open set $\Omega \subset X$, we set
\begin{align*}
& \cF_{\mbox{\tiny{{loc}}}}(\Omega)  :=  \{ f \in L^2_{\mbox{\tiny{loc}}}(\Omega) : \forall \textrm{ compact } K \subset \Omega, \ \exists f^{\sharp} \in D(\cE), f\big|_K = f^{\sharp}\big|_K \mbox{ $\mu$-a.e.} \},
\end{align*}
and we will write $\cF_{\mbox{\tiny{{loc}}}} =  \cF_{\mbox{\tiny{{loc}}}}(\cX)$.

We say that a function $u\in \cF$ is harmonic in $\Omega$ if
$$\cE(u, \varphi)=0\quad \text{for all}\; \varphi\in\cF^0(\Omega).$$
The {\em elliptic Harnack inequality} (EHI) holds if there exists a constant $C_H > 1$ and $\delta =\frac{1}{2}$ such that, for any ball 
$B(x,r)$ in $\cX$ and for any non-negative harmonic function $u$ on $B (x, r)$,
$$\esssup_{B(x, \delta r)} u \leq C_H\, \essinf_{B(x, \delta r)} u.$$
The choice of $\delta=\frac{1}{2}$ is arbitrary and could be replaced by any other parameter $\delta \in (0,1)$, as can be seen by a Whitney covering argument.

The {\em mean exit time estimate} $(\widetilde{E}_F)$ holds if
\begin{align*}
F(r) \lesssim \Exp_x \tau_{B(x,r)} \lesssim F(r)\,,
\end{align*}
for all $r>0$ and $x \in \cX \setminus \mathcal N$, where $\mathcal N$ is a properly exceptional set.

The {\em Faber-Krahn inequality} (FK) holds if there is a positive constant $\nu$ such that,
for all balls $B=B(x,r) \subset X$ and for all non-empty open sets $\Omega \subset B$, 
\begin{align*}
\lambda_0(\Omega) \gtrsim \frac{1}{F(r)} \left(\frac{\mu(B)}{\mu(\Omega)}\right)^\nu.
\end{align*}
where $\lambda_0(\Omega)$ is the bottom of the spectrum of the (positive) generator of
$(\cE,\cF^0(\Omega))$, that is,
$$ \lambda_0(\Omega) = \inf_{f \in \cF^0(\Omega) \setminus \{ 0 \}} \frac{\cE(f,f)}{\|f\|_2^2}. $$

Recall the following increasing function from \cite{GT12},
\begin{align} \label{Phi}
\Phi(s)=\sup_{r>0}\left\{\frac{s}{r}-\frac{1}{F(r)}\right\}.
\end{align}

The following result is part of \cite[Theorem 3.14]{GH14}.
\begin{theorem}[\cite{GH14}]
Assume (VD) and (RVD). Then (EHI) and $(\widetilde{E}_F)$ holds if and only if
the heat kernel $p_t(x,y)$ exists, has a H\"older continuous version in $x,y \in \cX$
, and there are constants $c,C \in (0,\infty)$ such that the following upper estimate holds,
\begin{equation}\label{UE}
p_t(x, y)\leq \frac{C}{\sV(x, \sR(t))} \exp\left(-\frac{t}{2}\Phi\left(c\frac{\D(x, y)}{t}\right)\right), \quad t>0,
\end{equation}
and there exist constants $\eta,c \in (0,1)$ such that the near-diagonal lower estimate holds,
\begin{equation}\label{NLE}
p_t(x, y)\geq \frac{c'}{\sV(x, \sR(t))},  \quad \forall t>0, \ \forall x,y \in \cX \mbox{ with } d(x, y)\leq \eta \sR(t).
\end{equation}
\end{theorem}

From here until the end of Section \ref{Sec-Uni} we assume that (VD), (RVD), (EHI) and $(\widetilde{E}_F)$ hold.

\subsection{Eigensolutions}\label{eigensol}
Most of the analysis done in this article depend on the Feynman-Kac representation of the eigenfunctions. To justify the representation, we cite here a large family of potentials for which
the Feynman-Kac representation holds.
Let $(\cE, \cF)$ be a regular Dirichlet form on $(\cX, \mu)$. Let $V$ be a bounded, Borel measurable function.  Let $\{P_x, X_t\}$ be the Hunt processes associated to $(\cE, \cF)$ and
$\cN$ is a properly exceptional set. Consider an open set 
$\Omega$ in $\cX$ such that $\mu(\Omega)<\infty$.
Let $\uptau$ be the exit time from $\Omega$. As before, denote by $(\cE, \cF^0(\Omega))$ the Dirichlet-type
 restriction of $(\cE, \cF)$ to $\Omega$.  Let us define the Feynman-Kac semigroup as
$$T^{V, \Omega}_t f(x) =\Exp_x\left[ e^{-\int_0^t V(X_s) ds} f(X_t) \Ind_{\{t<\uptau\}}\right], \quad x\in \Omega \setminus\cN, f\geq 0.$$
It then follows from \cite[Theorem 5.1.3]{CF12} that the above semigroup is $\mu$-symmetric and the corresponding Dirichlet form is given by
$$\cE^\nu(u, v)=\cE(u, v) + (u, v)_\nu,\quad u, v\in \cF\cap L^2(\Omega, \nu),$$
where $\nu$ is the Revuz measure associated to the potential $V$. Since $V$ is bounded we have 
$$(u, v)_\nu= \int_{\Omega} u v V d\mu.$$
We assume the heat kernel upper bound \eqref{UE} holds. Take $f\in L^1(\Omega, \mu)$. Then 
\begin{align*}
\norm{T^{V, \Omega}_t f}_\infty \leq e^{\norm{V}_\infty t}\int_{\Omega} |f(y)| \norm{p_t(x, \cdot)}_{\infty} d\mu\leq \sup_{x\in \cX} \frac{e^{\norm{V}_\infty t}}{\sV(x, \sR(t))} \norm{f}_{L^1(\Omega)}.
\end{align*}
Now if $\inf_{x\in\cX} \sV(x, \sR(t))>0$, by \cite[Theorem~2.1]{BBCK} there exists a heat kernel $q^V_t(x,y)$ for $T^{V, \Omega}_t$, i.e., $T^{V, \Omega}_t(x, dy) = q^V_t(x, y) d\mu(y)$ and moreover, 
$$q^V_t(x, y)\leq \sup_{x\in \cX}\frac{e^{\norm{V}_\infty t}}{\sV(x, \sR(t))}\quad \text{for all}\; t>0, \; x, y\in \Omega\setminus\cN.$$
From the above bound it is obvious that $q^V_t\in L^2(\Omega\times\Omega, \mu\times\mu)$. Therefore, $T^{V, \Omega}_t: L^2(\Omega)\to  L^2(\Omega)$ is a Hilbert-Schmidt
operator and hence compact. Also, $T^{V, \Omega}_t$ is a symmetric operator as claimed above. Therefore, there exists a countable family of Dirichlet eigenpairs $\{(\varphi_n, \lambda_n)\}$ such that
$\lambda_1<\lambda_2\leq\lambda_3\ldots \to \infty$ and $\{\varphi_n\}$ forms an orthonormal basis in $L^2(\Omega)$ satisfying (see \cite{GGK})
$$T^{V, \Omega}_t u = \sum_{n=1}^\infty e^{-\lambda_n t} (u, \varphi_n) \varphi_n, \quad u\in L^2(\Omega), \quad t\geq 0.$$
It is also routine to verify that $\varphi_1$ has a fixed sign in $\Omega$. It is evident that $\varphi_n$ is in the domain of the generator of $\left(T^{V, \Omega}_t\right)_{t>0}$ and 
$$\cE(\varphi_n, v) + (\varphi_n V, v)=\lambda_n (\varphi_n, v) \quad \text{holds for any}\; v\in \cF^0(\Omega).$$
One could also consider singular potentials for which $T^{V, \Omega}$ is a compact semigroup and therefore, the above theory applies. In this article we are interested in the 
solution of 
\begin{equation}\label{E2.1}
\cE(u, v) + (u V, v)=0\quad\text{for all}\; v\in\cF^0(\Omega),
\end{equation}
which should be understood as $T^{V, \Omega}_t u =u $ in $\Omega$ for all $t\geq 0$.

\subsection{Hitting time estimates}\label{S-Hitting}
In this section we obtain some hitting time estimate which will be essential for our analysis. Recall that the Green function is given by
$$G(x, y)=\int_0^\infty p_t(x, y) \D{t}, \quad x, y\in \cX.$$
We shall assume that $G(x, y)$ is continuous in $x, y$ for $x\neq y$. Also define
$$G_T(x, y) =\int_0^T p_t(x, y).$$ 

For $r>0$ we let
$$T=F(\eta' r),$$ 
where $\eta'=2\eta^{-1}$ and $\eta$ is the parameter in the near-diagonal lower bound \eqref{NLE}.
We also fix a reference point $o\in \cX \setminus\cN$. 

We begin with the following lemma.
\begin{lemma}\label{L2.1}
The following holds.
\begin{itemize}
\item[(a)] Suppose $B=B(x, r)$ be such that $\overline{B}\subset \cX \setminus \{y\}$, then, for any $\delta > 0$, we have
$$\sup_{z\in B(x, \delta r)} G(z, y)\leq C_H \, \inf_{z\in B(x, \delta r)} G(z, y).$$
\item[(b)] For $\D(x, y)\leq r$ we have
$$G(x, y)\gtrsim G_T(x,y) \gtrsim \frac{F(r)}{\sV(x, r)}.$$
\end{itemize}
\end{lemma}

\begin{proof}
Pick $\varepsilon$ small enough so that $B(y, \eps)\cap \overline{B}=\emptyset$. Define for $z\in B$, define
$$h_\eps(z, y)= \frac{1}{\sV(y, \eps)}\int_{B(y, \eps)} G(z, \xi) d\mu(\xi)= \frac{1}{\sV(y, \eps)}\Exp_z\left[\int_0^\infty \Ind_{B(y, \eps)}(X_t) dt\right].$$
For any Borel set $D$ in $B$, we note (by strong Markov property) that
$$h_\eps(z)=\Exp_z[h_\eps(X_{\uptau_D})].$$
Hence $h_\eps$ is harmonic in $B$ and therefore, it is also harmonic in the sense of a weak solution \cite[Chapter 6.7]{CF12}. Thus applying \cite{GH14}
we find 
$$\sup_{z\in B(x, \delta r)} h_\eps(z, y)\leq C_H \, \inf_{z\in B(x, \delta r)} h_\eps(z, y).$$
Now, using continuity of $G$ and letting $\eps\to 0$ we obtain (a).

Next we establish (b). Due to the near-diagonal lower bound \eqref{NLE}, we have for $r\geq d(x, y)$ that
\begin{align*}
G(x, y)\geq G_T(x, y) \geq \int_{F(\eta^{-1} r)}^{T} p_t(x, y) dt
 &\gtrsim \int_{F(\eta^{-1} r)}^{F(\eta' r)} \frac{1}{\sV(x, \sR(t))} dt\
\\
&\gtrsim \frac{F(\eta' r)-F(\eta^{-1} r)}{\sV(x, \eta' r)}
\\
&= \frac{F'(\xi r) r \eta^{-1}}{\sV(x, \eta' r)} \quad [\mbox{for some}\; \xi\in(\eta^{-1}, \eta')]
\\
&\gtrsim \frac{F(\xi r)}{\sV(x, \eta' r)} 
\\
&\gtrsim \frac{F(r)}{\sV(x,r)}, 
\end{align*}
where in the third line we used the mean-value theorem, the fourth line follows from \eqref{A1} and in the last line we used \eqref{beta} and \eqref{1.2}.
\end{proof}

The following estimate will be useful in the sequel.
\begin{lemma}\label{L2.2}
Let $r>0$. Recall that $T = F(\eta' r)$. Then 
\begin{itemize}
\item[(a)] we have  $$\int_0^{T} \frac{1}{V(x, \sR(t))} \exp\left(-\frac{t}{2}\Phi\left(c\frac{r}{t}\right)\right) dt
\lesssim \frac{F(r)}{\sV(x, r)}.$$

\item[(b)] for any point $x\in \cX$ we have
$$\sup_{z\in B(x, r)}\int_{B(x, r)} G_T(z, y) d\mu(y)\lesssim F(r).$$
\end{itemize}
\end{lemma}

\begin{proof}
Fix $r>0$. Then
\begin{align*}
\int_0^{F(\eta' r)} \frac{1}{\sV(x, \sR(t))} \exp\left(-\frac{t}{2}\Phi\left(c\frac{r}{t}\right)\right) dt
&= \int_0^{\eta' r} \frac{F'(s)}{\sV(x, s)} \exp\left(-\frac{F(s)}{2}\Phi\left(c\frac{r}{F(s)}\right)\right) ds\quad [\mbox{substituting }\; t=F(s)]
\\
&=\int_0^{\eta' r} \frac{sF'(s)}{s\sV(x, s)} \exp\left(-\frac{F(s)}{2}\Phi\left(c\frac{r}{F(s)}\right)\right) ds
\\
&\lesssim F(\eta' r) \int_0^{\eta' r} \frac{1}{s\sV(x, s)} \exp\left(-\frac{F(s)}{2}\Phi\left(c\frac{r}{F(s)}\right)\right) ds,\quad [\mbox{using}\; \eqref{A1}].
\end{align*}
For any $F(s)>0$ we note that
$$F(s)\Phi\left(\frac{r}{F(s)}\right)= F(s)\, \sup_{m>0}\left\{\frac{r}{F(s) m}-\frac{1}{F(m)}\right\}=\sup_{m>0}\left\{\frac{r}{m}-\frac{F(s)}{F(m)}\right\}\geq \frac{r}{s}-1, $$
choosing $m=s$. Using this in above expression we have
\begin{align*}
\int_0^{F(\eta' r)} \frac{1}{\sV(x, \sR(t))} \exp\left(-\frac{t}{2}\Phi(c\frac{r}{t})\right)&\lesssim F(\eta' r) \int_0^{\eta' r} \frac{1}{s\sV(x, s)} \exp\left(\frac{1}{2}-\frac{r}{2s}\right) ds\nonumber
\\
&=F(\eta' r) \int_{1/\eta'}^\infty \frac{1}{t\sV(x, r/t)} \exp\left(\frac{1}{2}-\frac{t}{2}\right) dt\quad [\mbox{substituting }\; t=r/s]\nonumber
\\
&\lesssim \frac{F(\eta' r)}{\sV(x, r\eta')} \int_{1/\eta'}^\infty t^{\alpha_2-1} \exp\left(\frac{1}{2}-\frac{t}{2}\right) dt\quad  [\mbox{Using } \; \eqref{1.2}]\nonumber
\\
&\lesssim \frac{F(\eta' r)}{\sV(x, r\eta')}
\\
&\lesssim \frac{F(r)}{\sV(x, r)}\quad [\text{by}\; \eqref{beta} \; \text{and}\; \eqref{1.2}].
\end{align*}
This completes the proof of (a).

Now we consider (b). Let $z \in B(x,r)$. Let $\uptau_{3r}$ be the exit time from the ball $B(z, 3r)\supset B(x, r)$. Then denoting $B(x, r)=B$ we write
\begin{align}\label{EL2.2A}
\int_{B} G_T(z, y) d\mu(y) &= \Exp_z\left[\int_0^T \Ind_{B}(X_t) dt\right]\nonumber
\\
&= \Exp_z\left[\Ind_{\{T\leq \uptau_{3r}\}}\int_0^T \Ind_{B}(X_t) dt\right] + \Exp_z\left[\Ind_{\{\uptau_{3r}<T\}}\int_0^T \Ind_{B}(X_t) dt\right]\nonumber
\\
&\leq \Exp_z[\uptau_{3r}] + \Exp_z\left[\Ind_{\{\uptau_{3r}<T\}}\int_{\uptau_{3r}}^T \Ind_{B}(X_t) dt\right]\nonumber
\\
&\lesssim F(3r) + \Exp_z\left[\Ind_{\{\uptau_{3r}<T\}}\int_{B} G_T(X_{\uptau_{3r}}, y) d\mu(y)\right],
\end{align}
where we applied $(\widetilde{E}_F)$ in the last inequality.
Since $X$ has continuous paths, we have $d(z, X_{\uptau_{3r}})=3r$ and thus $\D(X_{\uptau_{3r}}, y)\geq r$. Hence it follows from part (a) and \eqref{UE} that
\begin{align}\label{EL2.2B}
G_T(z, y) \lesssim \frac{F(r)}{\sV(z, r)}.
\end{align}
Also note that by \eqref{1.2}
$$\frac{\sV(x, r)}{\sV(z, r)}\lesssim \left(\frac{r+d(x, z)}{r}\right)^{\alpha_2} \lesssim 2^{\alpha_2}.$$
Thus using \eqref{beta} and \eqref{EL2.2B} in \eqref{EL2.2A} we obtain (b).
\end{proof}

Next lemma will be useful to obtain hitting time estimates.
\begin{lemma}\label{L2.3} 
Let $T=F(\eta' r)$. Then
\begin{equation}\label{EL2.3A}
\esssup_{z\in B(x, \varepsilon)} \frac{1}{\sV(y, \eps)}\int_{B(y, \varepsilon)} G_T(z, \xi) d\mu(\xi)\lesssim G(x, y) 
\end{equation}
for all $\eps$ small.
\end{lemma}

\begin{proof}
Suppose $\D(x, y)>3\varepsilon$. Then due to the elliptic Harnack inequality of Lemma \ref{L2.1}(a) and the symmetry of $G$ we find
$$\sup_{z\in B(x, \varepsilon)} \sup_{\xi\in B(y, \varepsilon)} G(z, \xi)
\leq C_H \sup_{z\in B(x, \varepsilon)} G(z, y)\leq C_H^2 G(x, y).$$
Since $G_T\leq G$, we have \eqref{EL2.3A} in this case. 

Next we consider the case $\D(x, y)\leq 3\varepsilon$.
To prove \eqref{EL2.3A} we consider the function
$$h_{T,\varepsilon}(z, y) = \int_{B(y, \varepsilon)} G_T(z, \xi) d\mu(\xi).$$
Pick $z\in B(x, \varepsilon)$ and let $\uptau$ be the exit time from $B(x, 10\varepsilon)$.  Then for $z\in B(x, r)\setminus \cN$,
we obtain
\begin{align}\label{EL2.3B}
h_{T,\varepsilon}(z, y) &\leq \left[\int_0^T \Exp_z[\Ind_{B(y, \varepsilon)}(X_s)] ds\right]\nonumber
\\
&= \Exp_z\left[\Ind_{\{T\leq \uptau\}}\int_0^T \Ind_{B(y, \eps)}(X_t) dt\right] + \Exp_z\left[\Ind_{\{\uptau<T\}}\int_0^T \Ind_{B(y, \eps)}(X_t) dt\right]\nonumber
\\
&\leq \Exp_z[\uptau_{B(z, 20\eps)}] + \Exp_z\left[\Ind_{\{\uptau<T\}}\int_{\uptau}^T \Ind_{B(y, \eps)}(X_t) dt\right]\nonumber
\\
&\leq F(20\eps) + \Exp_z\left[\Ind_{\{\uptau<T\}}\int_{B(y, \eps)} G_T(X_{\uptau}, \xi) d\mu(\xi)\right]\nonumber
\\
&\lesssim  F(3\eps)+  \Exp_z\left[\int_{B(y, \eps)} G_T(X_{\uptau}, \xi) d\mu(\xi)\right], 
\end{align}
where we used $(\widetilde{E}_F)$, \eqref{beta}.
Since $d(X_\uptau, y)> 3\eps$, we have \eqref{EL2.3A} which implies  
$$\Exp_z\left[\sV(y, \eps)^{-1} \int_{B(y, \eps)}G_T(X_{\uptau}, \xi) d\mu(\xi)\right] \lesssim \Exp_z[G(X_\uptau, y)].$$

For $d(x, y)\leq 3\eps$ we note that $v(z)=\Exp_z[h(X_\uptau)]$ is a harmonic function in $B(x, 8\eps)$. Hence we have
$$v(z)\leq C_H v(x).$$

Again for any Borel set $A$ we have
\begin{align*}
\int_A G(x, y)d\mu(y)= \Exp_x\left[\int_0^\infty \Ind_A(X_t) dt\right]&\geq \Exp_x\left[\int_\uptau^\infty \Ind_A(X_t) dt\right]
\\
&= \Exp_x\left[\Exp_{X_\uptau}\int_0^\infty \Ind_A(X_t) dt\right]
\\
&=\Exp_x \left[\int_A G(X_\uptau, y) d\mu(y)\right]= \int_A \Exp_x[G(X_\uptau, y)] d\mu(y).
\end{align*}
Since $A$ arbitrary, this of course, implies that $G(x, y)\geq \Exp_x[G(X_\uptau, y)]$. Thus putting these estimates in \eqref{EL2.3B} we obtain,
$$ \sV(y, \eps)^{-1} h_{T,\varepsilon}(z, y) \lesssim  \frac{F(3\eps)}{\sV(y, \eps)} +  G(x, y).$$
Applying Lemma~\ref{L2.1}(b) we get \eqref{EL2.3A}.
\end{proof}

Now we prove one of our main hitting time estimates. By $\uuptau_K$ we denote the first hitting time to a set $K$ i.e.,
$$\uuptau_K=\inf\{t>0\; : \; X_t\in K\}.$$
The following proof is inspired by \cite{MP}.

\begin{lemma}\label{L2.4}
Let $K$ be any compact subset of $B(o, r)$ not containing $o$, and let $\uuptau=\uuptau_K$ be the
hitting time of $K$. Then for $\nu=\frac{1}{\mu(K)}\Ind_K\mu$, we have
\begin{equation}\label{EL2.4A}
\Prob_o(\uuptau\leq T)\geq \left[2\int_{K}\int_{K}\frac{G_T(x, y)}{G_T(o, y)} d\nu(y) d\nu(x)\right]^{-1},
\end{equation}
where $T=F(\eta' r)$. In particular, we have
$$\Prob_o(\uuptau\leq T)\geq C_1 \frac{\mu(K)}{\sV(o, r)},$$
for some universal constant $C_1$.
\end{lemma}

\begin{proof}
Define
$$h_{T, \eps}(x, y)=\int_{B(y, \eps)} G_T(x, \xi) d\mu(\xi),\quad \text{and}\quad h^*_{T,\eps}(x, y)=\sup_{z\in B(x, \eps)} h_{T, \eps}(z, y).$$ 
Note that for $x\neq y, x, y\notin \cN$ we have
\begin{equation}\label{EL2.4B}
\lim_{\eps\to 0}\frac{1}{V(y, \eps)}h^*_\eps(x, y)=G_T(x, y).
\end{equation}
This can be easily checked from the heat kernel estimates. Note that by \cite[Theorem~5.11]{GT12} $p_t(\cdot, y)$ is H\"{older} continuous uniformly in $t\in [\kappa, T]$ for $y\in B \setminus\cN$.
Since $\int_0^\kappa p_t(x, y) dt$ can be made very small for $\kappa$ small, as $x\neq y$, we have the claim.

Now following \cite{MP} we consider
$$Z_\eps=\int_K \int_0^T \frac{1}{h_{T, \eps}(o, y)}\Ind_{\{X_t\in B(y, \eps)\}} dt d\nu(y).$$
Note that $\Exp_o[Z_\eps]=1$. By symmetry and the Markov property we easily get \cite[pp. 237]{MP}
\begin{align}\label{EL2.4C}
\Exp_o[Z^2_\eps]&=2 \Exp_o \int_0^T ds\int_s^T dt \int_K\int_K \frac{\Ind_{\{X_s\in B(x, \eps), X_t\in B(y, \eps)\}}}{h_{T, \eps}(o, y) h_{T, \eps}(o, x)}d\nu(x)\, d\nu(y)\nonumber
\\
&\leq 2 \Exp_o \int_0^T ds \int_K\int_K \frac{\Ind_{\{X_s\in B(x, \eps)\}} h^*_{T, \eps}(x, y)}{h_{T, \eps}(o, y) h_{T, \eps}(o, x)}d\nu(x)\, d\nu(y)\nonumber
\\
&= 2 \int_K\int_K \frac{ h^*_{T, \eps}(x, y)}{h_{T, \eps}(o, y)}d\nu(x)\, d\nu(y)
\end{align}
It follows from Lemma~\ref{L2.1}(b) that $\min_{y\in K} G_T(o, y)$ is positive, and therefore $\frac{1}{V(y, \eps)}h_{\eps}(0, y)$ is positive for all $\eps$ small. Again by Lemma~\ref{L2.3} we have
$$\frac{1}{V(y, \eps)}h^*_{T, \eps}(x, y)\lesssim G(x, y)\lesssim G_T(x, y) + \int_T^\infty \frac{1}{V(x, \sR(s))} ds.$$
By the near-diagonal lower estimate \eqref{NLE} we know that
$$p_t(x, y)\gtrsim \frac{1}{V(x, \sR(t))}, \quad \text{for}\; d(x, y)\leq \eta \sR(t).$$
Since the Green function is finite, it immediately implies 
$\int_T^\infty \frac{1}{V(x, \sR(s))} ds$ is finite uniformly in $x\in K$. 

Thus using Lemma~\ref{L2.2}(b), \eqref{EL2.4B} and dominated convergence theorem we can pass to the limit in \eqref{EL2.4C} to obtain 
\begin{equation}\label{EL2.4D}
\lim_{\eps\to 0}\int_K\int_K \frac{h^*_\eps(x, y)}{h_\eps(o, y)} d\nu(x)\, d\nu(y)=\int_K\int_K \frac{G_T(x, y)}{G_T(o, y)} d\nu(x)\, d\nu(y).
\end{equation}
Note that $\{Z_\eps>0\}=\{\exists t\in (0, T]\; \text{and}\; y\in K\; \text{such that}\; X_t\in B(y, \eps)\}$ is the event that the process visits the $\eps$-neighborhood of $K$ by time $T$. By the Cauchy-Schwarz inequality,
\begin{align} \label{EL2.4F}
\Prob_o(Z_\eps>0)\geq \frac{(\Exp_o Z_\eps)^2}{\Exp_o[Z^2_\eps]}=\frac{1}{\Exp_0[Z^2_\eps]}.
\end{align}
Observe that 
\begin{align} \label{EL2.4E}
\Prob_o(\uuptau_K\leq T)= \lim_{\eps\to 0} \Prob_o(Z_\eps>0).
\end{align}
Hence \eqref{EL2.4A} follows by putting together \eqref{EL2.4E}, \eqref{EL2.4F}, \eqref{EL2.4C} and \eqref{EL2.4D}.

To prove the second part it is enough to note from Lemma~\ref{L2.1}(b) and ~\ref{L2.2}(b) that
$$\sup_{x\in K}\int_{K} \frac{G_T(x, y)}{G_T(o, y)} d\mu(y)\lesssim \sV(o, r).$$
\end{proof}

We note that Lemma~\ref{L2.4} holds when the Hunt process is transient. Next, we obtain similar results for the recurrence case.
\begin{lemma}\label{L2.5}
Suppose that $F(r)=r^\beta$, $\sV(x, r)\simeq r^\alpha$ and $\beta\in [\alpha, 2\alpha)$. For $r>0$ we define $T=(\eta' r)^{\nicefrac{1}{\beta}}$.
Then there exists a universal constant $C_1$ such that for any compact $K\subset B(o, r)$, not containing $o$, we have
$$\Prob_o(\uuptau\leq T)\geq C_1 \frac{\mu(K)}{\sV(o, r)}.$$
\end{lemma}

\begin{proof}
From the calculations of Lemma~\ref{L2.1}(b) and Lemma~\ref{L2.1}(a) it is evident that
$$d(x, y)^{\beta-\alpha}\lesssim G_T(x, y)\lesssim T^{\frac{\beta-\alpha}{\beta}}\lesssim r^{\beta-\alpha} \quad \text{for}\; d(x, y)\leq r, \; \alpha<\beta.$$
It is also easy to check that $G_T(x, y)$ is continuous in both the variables. Hence we can justify the passage of limit in \eqref{EL2.4C} to obtain
\begin{equation}\label{EL2.5A}
\Prob_o(\uuptau\leq T)\geq \left[2\int_{K}\int_{K}\frac{G_T(x, y)}{G_T(o, y)} d\nu(y) d\nu(x)\right]^{-1}.
\end{equation}
Note that 
$$\iint \frac{G_T(x, y)}{G_T(o, y)}d\nu(x)\, d\nu(y) \lesssim \int_{K}\int_{K}\frac{r^{\beta-\alpha}}{d^{\beta-\alpha}(o, y)} d\nu(y) d\nu(x)=
\int_{K}\frac{r^{\beta-\alpha}}{\D^{\beta-\alpha}(o, y)} d\nu(y).$$
On the other hand, using the Lorentz-H\"older inequality \cite[Theorem~3.5]{oneil},
$$\left|\int_B d^{\alpha-\beta}(o, y) d\mu\right|\leq \norm{\Ind_B}_{L^{\frac{\alpha}{2\alpha-\beta}, 1}}\, \norm{\D^{\alpha-\beta}(o, y)}_{L^{\frac{\alpha}{\beta-\alpha}, \infty}}$$
Note that 
$$\norm{\Ind_B}_{L^{\frac{\alpha}{2\alpha-\beta}, 1}}\lesssim r^{2\alpha-\beta}$$
and 
$$\norm{\D^{\alpha-\beta}(o, y)}_{L^{\frac{\alpha}{\beta-\alpha}, \infty}}\leq C_2,$$
for some universal constant $C_2$. Therefore,
$$\int_{K}\frac{r^{\beta-\alpha}}{\D^{\beta-\alpha}(o, y)} d\nu(y)\lesssim \frac{1}{\mu(K)} r^{\alpha}\simeq \frac{\sV(o, r)}{\mu(K)}.$$
inserting this estimate in \eqref{EL2.5A} we have Lemma~\ref{L2.5}.

To complete the proof we remain to study the case $\alpha=\beta$. In this situation we have for $d(x, y)\leq r$ that
\begin{align*}
G_T(x, y)
&\lesssim \int_0^{\eta' r} \frac{s^{\beta-1}}{\sV(x, s)} \exp(-c d^{\nicefrac{\beta}{\beta-1}}(x, y) s^{-\nicefrac{\beta}{\beta-1}})\, ds \quad [\mbox{Substituting}\; s^\beta=t]
\\
&\lesssim \int_0^{\eta' r} \frac{1}{s} \exp(-c d^{\nicefrac{\beta}{\beta-1}}(x, y) s^{-\nicefrac{\beta}{\beta-1}})\, ds
\\
&\simeq \int_{\frac{d(x, y)}{\eta' r}}^\infty t^{-1}  \exp(-c t^{\nicefrac{\beta}{\beta-1}})\, dt \quad [\mbox{Substituting}\; d(x, y)/s=t].
\\
&\leq \int_{\frac{d(x, y)}{\eta' r}}^1 t^{-1}  dt + \int_{1}^\infty   \exp(-c t^{\nicefrac{\beta}{\beta-1}})\, dt
\\
&\lesssim (\log\frac{r}{d(x, y)} + 1).
\end{align*}
Similarly, since $\eta'=2\eta^{-1}$
$$G_T(x, y)\gtrsim \int_{\frac{d(x, y)^\beta}{\eta^\beta}}^{(\eta' r)^\beta} s^{-1} ds\gtrsim \log\frac{2^\beta r}{d(x, y)}.$$
Let us first complete the proof assuming that we arrive at \eqref{EL2.5A} in this case. Since $\log x\leq x$ for $x>0$, we note that
$$G_T(x, y)\lesssim \frac{r}{d(x, y)}\quad \text{for}\; d(x, y)\leq r.$$
Therefore, applying the Lorentz-H\"older inequality \cite[Theorem~3.5]{oneil} we get
\begin{align*}
\int_{B(o, r)} G_T(x, y) d\mu(y)\lesssim r \int_{B(o, r)} \frac{1}{d(x, y)} d\mu(y)
\leq  r \norm{\Ind_B}_{L^{\frac{\beta}{\beta-1}, 1}}\, \norm{\frac{1}{d(x, \cdot)}}_{L^{\beta, \infty}}\lesssim r^{\beta}\simeq \sV(o, r)\,.
\end{align*}
Since $G_T(o, y)\gtrsim 1$, we complete the proof.

Thus we remain to prove \eqref{EL2.5A}. To do this we have to justify the passage of limit in \eqref{EL2.4C}. We show a variant of Lemma~\ref{L2.3} which allows us to use dominated convergence theorem to pass to the limit in \eqref{EL2.4C}. We claim that
\begin{equation}\label{EL2.5B}
\esssup_{z\in B(x, \varepsilon)\cap K} \frac{1}{\sV(y, \eps)}\int_{B(y, \varepsilon)} G_T(z, \xi) d\mu(\xi)\leq \kappa_r\, G_T(x, y) 
\end{equation}
for all $\eps$ small and for some constant $\kappa_r$, depending on $r$.

Given a ball $B=B(x_0, R)$ in $\cX$ let $g^{B}$ be
the Green function in $B(x_0, R)$. From \cite[Theorem~3.12]{GH14} there exists a constant $K>1$, independent of $R, x_0$ such that
$$ \log\frac{RK}{d(x_0, y)} \lesssim g^{B}(x_0, y)\lesssim \log\frac{RK}{d(x_0, y)}, 
\quad \text{for all}\; y\in B(x_0, K^{-1}R)\setminus\{x_0\}.$$
It is also known that the Green function $g^B$ is  symmetric \cite[Lemma~5.2]{GH14}. Now fix $R=2rK$. By \cite[Lemma~3.2]{GT12}, we have $\lambda_{B(x_0, R)}\gtrsim \frac{1}{R^\beta}>0$.
Therefore, by \cite[Lemma~5.2]{GH14}, $g^{B}$ has Harnack property.  So we choose $\eps\in (0, r/3)$ and $r\geq d(x, y)\geq 3\eps$.
Then
\begin{align*}
\sup_{\xi\in B(y, \varepsilon)}\sup_{z\in B(x, \varepsilon)}G_T(z, \xi)& \lesssim \sup_{\xi\in B(y, \varepsilon)}\sup_{z\in B(x, \varepsilon)}(\log\frac{r}{d(z, \xi)} + 1)
\\
&\lesssim \sup_{\xi\in B(y, \varepsilon)}\sup_{z\in B(x, \varepsilon)} g^{B(\xi, R)}(z, \xi)
\\
&\leq C_H \sup_{\xi\in B(y, \varepsilon)} g^{B(\xi, R)}(x, \xi)
\\
&\leq C_H \sup_{\xi\in B(y, \varepsilon)} (\log\frac{r}{d(x, \xi)} + 1)
\\
&\lesssim C_H \sup_{\xi\in B(y, \varepsilon)} g^{B(x, R)}(\xi, x)
\\
&\leq C^2_H\, g^{B(x, R)}(y, x)\lesssim G_T(x, y).
\end{align*}
Thus \eqref{EL2.5B} holds when $r\geq d(x, y)\geq 3\eps$. There is nothing to prove if $d(x, y)\in( r, 2r)$. Now suppose $d(x, y)<3\eps$. Then from the proof of Lemma~\ref{L2.3} we obtain
\begin{align*}
h_{T, \eps}(z, y) &\lesssim \sV(y, 3\eps)^{-1} F(3\eps)+  \sV(y, \eps)^{-1}\Exp_z\left[\int_{B(y, \eps)} G_T(X_{\uptau}, \xi) d\mu(\xi)\right]
\\
&\lesssim 1 +  \sV(y, \eps)^{-1}\Exp_z\left[\int_{B(y, \eps)} G_T(X_{\uptau}, \xi) d\mu(\xi)\right].
\end{align*}
Note that $d(X_{\uptau}, \xi)\geq 3\eps$, and therefore, $G_T(X_{\uptau}, \xi)\lesssim G_T(x, y)$. Hence we arrive at \eqref{EL2.5B}. This completes the proof.
\end{proof}

\begin{remark}
 Note that the conditions in Lemma~\ref{L2.5} are satisfied by a large family of fractals. For instance, if $\cX$ is the unbounded Sierpi\'nski gasket in $\Rd, d\geq 2,$ then we have
$\alpha=\frac{\log(d+1)}{\log 2}$ and $\beta=\frac{\log(d+3)}{\log 2}$ (see \cite{BP, K01}). In case of Sierpi\'nski carpets when the \textit{spectral dimension}
$d_s$ is strictly larger than $2$ the conditions of Lemma~\ref{L2.5} are met \cite{BB99}.
\end{remark}

Our next result will be useful to get a bound on the decay of positive super-solutions.

\begin{lemma}\label{L2.6}
Grant the hypotheses of Lemma~\ref{L2.4}. Then the following holds.
\begin{itemize}
\item[(a)] For any $x\in\cX\setminus\cN$ with $d(o, x)\geq 2$ we have
$$\Prob_x(\uuptau_1<\infty)\gtrsim \frac{F (\theta)}{\sV(o, \theta)}, \quad \text{for any} \; \theta\geq d(o, x),$$
where $\uuptau_1$ denotes the hitting time to the ball $\overline{B(o, 1)}$.
\item[(b)] Suppose $x\in B(o, 2r)\setminus \left(\overline{B(o, r)}\cup\cN\right)$. There exists $c>0$, independent of $r$ and $x$, such that
$$\Prob_x(\uuptau_r<\infty)\geq c.$$
\end{itemize}
\end{lemma}

\begin{proof}
For (a) we fix $r=2+\theta\geq 2+d(o, x)$ and $K=\overline{B(o, 1)}$ in Lemma~\ref{L2.4}. Note that $x$ plays the role of $o$ here.
Then by Lemma~\ref{L2.1}(b) and \eqref{1.2} we have 
$$G_T(x, y)\gtrsim \frac{F(r)}{\sV(x, r)}\gtrsim \frac{F(\theta)}{\sV(x, \theta)}, \quad \forall\; y\in B(o, 1).$$
Also note that by \eqref{1.2},
$$\frac{\sV(x, \theta)}{\sV(o, \theta)}\leq C\left(\frac{\theta + \D(o, x)}{\theta}\right)^{\alpha_2}= C 2^{\alpha_2}.$$
Let $\uptau_{3}$ be the exit time from the ball $B(z, 3)\supset K$ where $z\in K=\overline{B(o, 1)}$. $G$ being continuous, we also have
$$\sup_{y\in K}\sup_{z'\in\partial B(z, 3), z\in K} G_T(z', y)\leq \sup_{y\in K}\sup_{z'\in\partial B(z, 3), z\in K} G(z', y)\leq C'.$$
We write
\begin{align*}
\int_{K} G_T(z, y) d\mu(y) &= \Exp_z\left[\int_0^T \Ind_{K}(X_t) dt\right]
\\
&= \Exp_z\left[\Ind_{\{T\leq \uptau_{3}\}}\int_0^T \Ind_{K}(X_t) dt\right] + \Exp_z\left[\Ind_{\{\uptau_{3}<T\}}\int_0^T \Ind_{K}(X_t) dt\right]
\\
&\leq \Exp_z[\uptau_{3}] + \Exp_x\left[\Ind_{\{\uptau_{3}<T\}}\int_{\uptau_{3}}^T \Ind_{K}(X_t) dt\right]
\\
&\lesssim F(3) + \Exp_z\left[\Ind_{\{\uptau_{3}<T\}}\int_{K} G_T(X_{\uptau_{3}}, y) d\mu(y)\right]
\\
& \lesssim F(3) + C' \sV(o, 1),
\end{align*}
where we used $(\widetilde{E}_F)$.
Putting these estimates in \eqref{EL2.4A} we get (a).

Now we come to (b). We fix $K=\overline{B(o, r)}$, $B=B(x, 4r)$, and $T=F(\eta' 4r)$. By Lemma~\ref{L2.1}(b) and \eqref{1.2} we have 
$$G_T(x, y)\gtrsim \frac{F(4 r)}{\sV(x, 4r)} \gtrsim \frac{F(4 r)}{\sV(o, 4r)}\gtrsim \frac{F(r)}{\sV(o, r)} \quad \forall\; y\in K.$$
By Lemma~\ref{L2.2}(b) and \eqref{beta},
\begin{align*}
\sup_{z\in K} \int_{K} G_T(z, y) d\mu(y) &\leq \sup_{z\in K} \int_{B(x, 4r)} G_T(z, y) d\mu(y) \lesssim F(4r)\lesssim F(r).
\end{align*}
This, of course, implies 
$$\sup_{z\in K}\; \int_{K} G_T(z, y) d\nu(y)\lesssim \frac{F(r)}{\sV(o, r)}.$$
Putting these estimates into \eqref{EL2.4A} we get (b).
\end{proof}

\section{Generalized Lieb's inequality}\label{Sec-GL}
In this section we prove several estimates of eigenvalues and nodal domains using the estimates we obtained in the previous section. We fix a domain $\Omega \subset \cX$ and
our central object of study are the non-zero solutions $u$ of (see \eqref{E2.1})
$$\cE(u, v) + (uV, v)=0\quad \text{for all}\; v\in\cF^0(\Omega).$$
Recall that our standing assumptions are (VD), (RVD), (EHI) and $(\widetilde{E}_F)$.
Also recall from Section~\ref{eigensol} that we always assume $u$ to have a Feynman-Kac representation i.e.
\begin{equation}\label{E3.1}
u(x)=T^{V, \Omega}_t u(x)=\Exp_x\left[ e^{-\int_0^t V(X_s) ds} u(X_t) \Ind_{\{t<\uptau\}}\right], \quad t\geq 0, \quad x\in\Omega\setminus\cN,
\end{equation}
where $\uptau=\uptau_\Omega$ denote the exit time from $\Omega$.

We begin with the Feynman-Kac representation of the eigenfunction in a smaller sub-domain.
\begin{lemma}\label{L3.1}
Let $u$ be a bounded, continuous function in $\Omega$ satisfying \eqref{E3.1}.
Let $\Omega_1\subset\Omega$ be open and $\uptau_1$ be the exit time from $\Omega_1$. Then we have for $x\in\Omega_1\setminus\cN$
$$u(x)=\Exp_x\left[e^{-\int_0^{t\wedge\uptau_1} V(X_s)ds} u(X_{t\wedge\uptau_1})\Ind_{\{t\wedge\uptau_1<\uptau_\Omega\}}\right].$$
\end{lemma}

\begin{proof}
Let $(X_t, \Prob_x, \sF_t)$ be the underlying Hunt processes with right continuous filtration $\sF_t$. Fix $x\in\Omega_1\setminus\cN$
and define
$$A_t= e^{-\int_0^t V(X_s) ds} u(X_t)\Ind_{\{t<\uptau_\Omega\}}, \quad t\geq 0.$$
We claim that $\{A_t, \sF_{t\wedge\uptau}, \Prob_x\}$ is a martingale where $\uptau=\uptau_\Omega$. In particular, note that for $s<t$ and using Markov property we have
\begin{align*}
\Exp_x\left[e^{-\int_0^t V(X_s) ds} u(X_t)\Ind_{\{t<\uptau\}}\Big| \sF_{s\wedge\uptau}\right] & = \Exp_x\left[e^{-\int_0^t V(X_s) ds} u(X_t)\Ind_{\{s<\uptau\}}\Ind_{\{t<\uptau\}}\Big| \sF_{s\wedge\uptau}\right]
\\
&= e^{-\int_0^s V(X_p) dp} \Ind_{\{s<\uptau\}}\Exp_{X_s}\left[e^{-\int_0^{t-s} V(X_p) dp}u(X_{t-s})\Ind_{\{t-s<\uptau\}}\right]
\\
&= e^{-\int_0^s V(X_p) dp} u(X_s) \Ind_{\{s<\uptau\}}=A_s
\end{align*}
Therefore, the claimed result follows from Doob's optimal  stopping theorem.
\end{proof}

The next result on the eigenvalue bound is well-known \cite{GH14, HeSC01}.
\begin{lemma}\label{L3.2}
Let $B$ be a ball of radius $r$ with center at $x$ and let $\lambda_B$ be its principal eigenvalue. Then
$$\lambda_B\lesssim \frac{1}{F(r)}\quad \text{for all}\; r>0.$$
\end{lemma}

\begin{proof}
Recall that there exists a function $e_B\in\cF(B(x, Kr))$ such that $e_B=1$ almost surely on $B$ and 
$$\cp(B, B(x, Kr))=\cE(e_B, e_B).$$
Therefore,
$$\lambda_B\leq \frac{\cE(e_B, e_B)}{(e_B, e_B)}\leq \frac{\cp(B, B(x, Kr))}{\mu(B)}.$$
On the other hand, by \cite[Theorem~3.12]{GH14}, the resistance condition holds which implies
$$\cp(B, B(x, Kr))\simeq \frac{\mu(B)}{F(r)}.$$
This proves the result.
\end{proof}

Using Lemma~\ref{L3.1} and ~\ref{L3.2}, we establish the following \textit{wavelength density} result.

\begin{theorem}\label{T3.1}
There exists a universal constant $C$, not depending on $(\Omega, \lambda)$ such that for any bounded, continuous Dirichlet eigenpair $(u, \lambda)$, $\lambda>0$,
the eigenfunction $u$ must vanish somewhere in any ball $B(x, \sR(C\lambda^{-1}))$ that is contained in $\Omega$.
\end{theorem}

\begin{proof}
We begin with a $C$ and later we fix the choice of $C$. Let $B(x, \sR(C\lambda^{-1}))\Subset\Omega$ be such that $u$ is strictly positive in 
$\overline{B(x, \sR(C\lambda^{-1}))}$. Let $\lambda_B$ be the principal eigenvalue in the ball with principal eigenfunction $u_B$ i.e.,
$$u_B(x)=\Exp_x[e^{\lambda_B t}u_B(X_t)\Ind_{\{t<\uptau_B\}}].$$
Thus we can find a point $x_0$ in $B\setminus\cN$ such that $u_B(x_0)>0$ and using the above representation we arrive at
\begin{equation}\label{T3.1A}
\lambda_B+ \liminf_{t\to\infty}\frac{1}{t}\Prob_{x_0}(t<\uptau_B)\geq 0.
\end{equation}
Here $\uptau_B$ denotes the exit time from $B$. Applying Lemma ~\ref{L3.1} we obtain
$$u(x_0)\geq \Exp_{x_0}\left[e^{\lambda t} u(X_{t})\Ind_{\{t<\uptau_B\}}\right]\geq [\min_{\overline{B}} u]\, e^{\lambda t } \Prob_{x_0} (t<\uptau_B).$$
Taking logarithm on both sides, dividing by $t$, and letting $t\to\infty$ and using \eqref{T3.1A} we get
$$0\geq \lambda + \liminf_{t\to\infty}\frac{1}{t}\Prob_{x_0}(t<\uptau_B)\geq \lambda-\lambda_B.$$
Now apply Lemma~\ref{L3.2} to find
$$\lambda\leq C_1 \frac{1}{F(\sR(C\lambda^{-1}))}\leq \frac{C_1}{C}\lambda.$$
To get a contradiction choose $C>C_1$. Hence the proof.
\end{proof}

From Theorem~\ref{T3.1} we note that the nodal domains of the eigenfunction corresponding to the eigenvalue $\lambda$ can not have inradius
$\sR(C\lambda^{-1})$. However, one can use the Faber-Krahn inequality to give a lower bound on the measure of the nodal domains. In this spirit
we generalize a famous result of Lieb \cite{L83} in our next main result.

Recall that $\eta$ is the parameter in the near-diagonal lower bound \eqref{NLE}.
\begin{theorem}\label{T3.2}
Suppose $V$ is bounded and $u\neq 0$ satisfies \eqref{E3.1}. Then for any given $\varepsilon\in (0, 1)$,
there exists a constant $\kappa >0$ such that for $r=\frac{\eta}{2}\sR(\kappa \norm{V^-}_{L^\infty(\Omega)}^{-1})$ and for some point $o\in \Omega$ we have 
$$\mu(\Omega \cap B(o, r))\geq (1-\varepsilon)\mu(B(o, r)).$$
In addition, if we assume $u$ in $C(\overline{\Omega})$ and $\cN=\emptyset$ then we can choose $o$ to be any maximizer of $|u|$.
\end{theorem}

\begin{proof}
By definition,
$$T^{V, \Omega}_t u(x) =\Exp_x\left[ e^{\int_0^t -V(X_s) ds} u(X_t) \Ind_{\{t<\uptau_\Omega \}}\right]=u(x), \quad \text{a.e. in}\; x.$$
Therefore, denoting $\Theta=\norm{V^-}_{L^\infty(\Omega)}$ we obtain
\begin{align*}
\int_\Omega u^2 d\mu &= \int_{\Omega} (T^{V, \Omega}_tu(x))^2 d\mu(x)
\\
&\leq \int_{\Omega} \Exp_x\left[ e^{\int_0^t -2V(X_s) ds} \Ind_{\{t<\uptau_\Omega \}}\right] \Exp_x[u^2(X_t)] d\mu(x)
\\
&\leq e^{2t\Theta}\esssup_{\Omega\setminus\cN}\Prob_x(t<\uptau_\Omega) \int_{\Omega} \Exp_x[u^2(X_t)] d\mu(x)
\\
&\leq e^{2t\Theta}\esssup_{\Omega\setminus\cN}\Prob_x(t<\uptau_\Omega) \int_{\Omega}\int_{\cX} u^2(y)p_t(x,y) d\mu(y)\, d\mu(x)
\\
&\leq e^{2t\Theta} \esssup_{\Omega\setminus\cN}\Prob_x(t<\uptau_\Omega) \int_\Omega u^2(y) d\mu(y).
\end{align*}
Hence we obtain
$$1\leq e^{2t\Theta} \esssup_{\Omega\setminus\cN}\Prob_x(t<\uptau_\Omega),$$
for all $t>0$, which in turn implies,
$$e^{-2t}\leq \esssup_{\Omega\setminus\cN} \Prob_x(t\Theta^{-1}<\uptau_\Omega).$$
Pick any $\delta>1$. From above we can find a point $o\in \Omega\setminus\cN$ satisfying 
$$\frac{1}{\delta}e^{-2t} \leq \Prob_o(t\Theta^{-1}<\uptau_\Omega),$$
hence
\begin{equation}\label{ET3.2A}
 \Prob_o(t \Theta^{-1} \geq \uptau_\Omega)\leq 1-\frac{1}{\delta}e^{-2t}.
\end{equation}
Note that the choice of $o$ depends on $t$ and $\delta$. We shall fix our choice of $t$ and $\delta$ later, depending on $\eps$.
Let $r= \frac{\eta}{2}\sR(t\Theta^{-1})$. Let $E=\Omega^c\cap B(o, r)$ and $T=t \Theta^{-1}$. We claim that, for some universal constant $\kappa_1$, independent of
$t,\, \delta,\, o,\, V,\, \Omega$, we have
\begin{equation}\label{ET3.2B}
\frac{\mu(E)}{\mu(B(o, r))}\leq \kappa_1 \Prob_{o}(\uptau_\Omega\leq T).
\end{equation}
If $\mu(E)=0$ then there is nothing to prove. So we assume $\mu(E)>0$ and choose $K\subset E$ compact, such that $2\mu(K)\geq \mu(E)$. Then \eqref{ET3.2B} follows
from Lemma~\ref{L2.4}.
Combining \eqref{ET3.2B} with \eqref{ET3.2A} we get
$$\frac{\mu(E)}{\mu(B(o, r))}\leq \kappa_1\left(1-\frac{1}{\delta}e^{-2t}\right), \quad t>0\,.$$
Now we choose $t=t_\eps$ and $\delta=\delta_\eps$ suitably so that 
$$\kappa_1 \left(1-\frac{1}{\delta_\eps}e^{-2t_\eps} \right)<\eps.$$
This yields
$$\mu(\Omega\cap B(o, r))\geq (1-\eps)\mu(B(o, r)).$$
\end{proof}

In the remaining part of this section we will generalize Theorem~\ref{T3.2} to a more general class of potentials, possibly singular.
To do so, we restrict ourselves to the case where $F(r)=r^\beta$ and $\sV(x, r)\simeq r^\alpha$.
We need Mittag-Leffler functions \cite{GKMR}
$$
\ML_{\ell}(x)=\sum_{k=0}^\infty \frac{ x^k}{\Gamma(1+\ell k)},\quad x\geq 0.
$$
We begin with the following Khasminiskii type lemma.

\begin{lemma}\label{L3.3}
Let $F(r)=r^\beta$ and $\sV(x, r)\simeq r^\alpha$.
Suppose that $p>\frac{\alpha}{\beta}\vee1$. Then for any $V\geq 0$, supported in $\Omega$ we have the following estimate
$$\sup_{x\in\Omega\setminus \cN}\Exp_x\left[e^{\int_0^t V(X_s) ds}\right]\;\leq\;  \ML_\varrho(\kappa_p \norm{V}_{p, \Omega} t^\varrho),$$
for $\varrho=1-\frac{\alpha}{\beta p}$, where $\kappa_p>1$ is a constant not depending on $\Omega$. In particular,
$$\sup_{x\in\Omega\setminus \cN}\Exp_x\left[e^{\int_0^t V(X_s) ds}\right]\;\leq\; \kappa_p e^{\left(\kappa_p \norm{V}_{p, \Omega}\right)^{\nicefrac{1}{\varrho}} t}.$$
\end{lemma}

\begin{proof}
From the heat kernel upper bound \eqref{UE} it is evident that for some constant $C$ we have
$$p_t(x, y)\leq \frac{C}{t^{\nicefrac{\alpha}{\beta}}}, \quad t>0,\quad  x, y\in \cX.$$
Then, for any $s$ we have
\begin{align*}
\Exp_x[V(X_s)] &\leq \norm{V}_p\left[\int_{\cX} (p_s(x, y))^{p'} d\mu(y)\right]^{\nicefrac{1}{p'}}
\\
&\leq C^{1/p} \norm{V}_p s^{-\frac{\alpha}{p\beta}} \left[\int_{\cX} (p_s(x, y)) d\mu(y)\right]^{\nicefrac{1}{p'}}\leq C^{1/p} \norm{V}_p s^{-\frac{\alpha}{p\beta}}.
\end{align*}
Define $C_1=C^{1/p}$. Therefore, for $0< s_1< s_2< s_3\ldots< s_k$, using Markov property we get that
\begin{eqnarray*}
\Exp_x[V(X_{s_1})\cdots V(X_{s_k})]
&=&
\Exp_x[V(X_{s_1})\cdots V(X_{s_{k-1}}) \Exp_x[V(X_{s_k})\,|\, \sF_{s_{k-1}}]] \\
&=&
\Exp_x[V(X_{s_1})\cdots V(X_{s_{k-1}}) \Exp_{X_{s_k}}[V(X_{s_k-s_{k-1}})] \\
&\leq&
C_1 \norm{V}_{p} (s_k-s_{k-1})^{-\frac{\alpha}{\beta p}} \Exp_x[V(X_{s_1})\cdots V(X_{s_{k-1}})] \\
&\leq &
\!\!\! \ldots \leq (C_1 \norm{V}_{p})^k s_1^{-\frac{\alpha}{\beta p}} (s_2-s_1)^{-\frac{\alpha}{\beta p}}
\cdots (s_k - s_{k-1})^{-\frac{\alpha}{\beta p}}.
\end{eqnarray*}
Hence (see also \cite[Lemma~4.51]{feyn} in the second edition)
\begin{eqnarray*}
\lefteqn{
\Exp_x\left[\frac{1}{k!}\left(\int_0^t V(X_s)\, \D{s}\right)^k\right] } \\
&\leq \;&
\int_0^t d{s_1} \int_{s_1}^t d{s_2} ... \int_{s_k}^t d{s_k} \, \Exp_x [V(X_{s_1})V(X_{s_2})...V(X_{s_k})] \\
&\leq \;&
C_1^k \norm{V}_{p}^k \int_0^t d{s_1}\int_{s_1}^t d{s_2} ... \int_{s_k}^t d{s_k} \,
s_1^{-\frac{\alpha}{\beta p}} (s_2-s_1)^{-\frac{\alpha}{\beta p}} \cdots (s_k - s_{k-1})^{-\frac{\alpha}{\beta p}} \\
&=&
\frac{\left(C_1\norm{V}_{p}t^\varrho\Gamma(\varrho)\right)^k}{\Gamma(1+k\eta)}, \quad t\leq \kappa_1,
\end{eqnarray*}
where $\varrho=1-\frac{\alpha}{\beta p}>0$, by our choice of $p$. Summing over $k$ and using the Mittag-Leffler function, we get
which gives
\begin{equation}\label{2.20}
\sup_{x\in \Omega\setminus\cN}\Exp_x\left[e^{\int_0^t V(X_s)\, \D{s}}\right]\leq \ML_\eta (C_1\norm{V}_{p}t^\varrho\Gamma(\varrho)).
\end{equation}
This gives the first part of the result.
It is also known that for some constant $m_\varrho$, dependent only on $\varrho$, 
$$
\ML_\varrho(x)\leq m_\varrho e^{x^{\nicefrac{1}{\varrho}}}, \quad  x\geq 0,
$$
holds. Thus using this estimate in \eqref{2.20} we have second part of the proof.
\end{proof}

Using Lemma~\ref{L3.3} and repeating the arguments of Theorem~\ref{T3.2}, we may improve the result by replacing $\norm{V^-}_{\infty}$ by $\norm{V^-}_p$.

\begin{theorem}\label{T3.3}
Grant the hypotheses of Lemma~\ref{L3.3}.
Suppose that there exists an eigenfunction $u$ in $\Omega$ with potential $V$ satisfying \eqref{E3.1}. Consider $p>\frac{\alpha}{\beta}\vee 1$. Then for every $\eps\in (0, 1)$
there exists a constant $\kappa>0$, independent of $\Omega, V, u$, such that for $r=\kappa \norm{V^-}^{-\frac{1}{\beta\varrho}}_{p, \Omega}$ and for some point $o\in \Omega$,
$$\mu (\Omega\cap B(o,r))\geq (1-\eps)\mu(B(o,r)).$$
Moreover, if $u\in \Cc(\overline\Omega)$ and $\cN=\emptyset$, then we can choose $o$ as a maximizer of $|u|$.
\end{theorem}

As a consequence to the above result we have the following. The first one generalizes \cite{DCH, DEH} where similar results are obtained
in the Euclidean setting.

\begin{corollary}\label{C2.1}
Assume the setting of Theorem~\ref{T3.3}. Then there exists a universal constant $c$, not depending on $\Omega$, $V$, satisfying
$$\mu(\Omega)^{\frac{\beta}{\alpha}-\frac{1}{p}}\norm{V^-}_{p, \Omega}\geq c.$$
\end{corollary}

\begin{proof}
Fix $\eps=\frac{1}{2}$ in Theorem~\ref{T3.3} above and use the fact $\frac{1}{\beta\varrho}=\frac{p}{p\beta-\alpha}$.
\end{proof}

The next result generalizes \cite{DEL} where a moment estimate on the negative principal eigenvalue was obtained on spheres. It can also be seen as
a generalization of the classical Keller's inequality in bounded domains. 

\begin{corollary}\label{C3.2}
For the universal constant $c$, same as in Corollary~\ref{C2.1}, there exists no non-positive eigenvalue if 
$$\mu(\Omega)^{\frac{\beta}{\alpha}-\frac{1}{p}}\norm{V^-}_{p, \Omega}< c.$$
Moreover, if $\lambda_\Omega$ is the non-positive eigenvalue then we have
$$\abs{\lambda_\Omega}^{\eta}\leq c_p \norm{V^-}_{p, \Omega},$$
for some universal constant $c_p$ not depending on $V$, $\Omega$, where $\eta=1-\frac{\alpha}{\beta p}>0$.
\end{corollary}

\begin{proof}
The first part follows from Corollary~\ref{C2.1} since $(V-\lambda_\Omega)^-\leq V^-$. Since
$$T^{V, \Omega}_t u(x) =\Exp_x\left[ e^{\int_0^t V(X_s) ds} u(X_t) \Ind_{\{t<\uptau_\Omega \}}\right]=e^{-\lambda_\Omega t}u(x), \quad \text{a.e. in}\; x,$$
from the proof of Theorem~\ref{T3.2} we find that
$$e^{-2\lambda_\Omega t}\leq \esssup_{\Omega\setminus\cN} \Exp_x\left[ e^{\int_0^t 4V^-(X_s) ds}\right].$$
Then the proof follows applying the second part of Lemma~\ref{L3.3}.
\end{proof}

\section{Uniqueness of supersolutions}\label{Sec-Uni}

The chief goal of this section is to provide sufficient conditions for the uniqueness of super-solutions to certain type of  semi-linear partial differential equations.
More specifically, we are interested in the non-negative weak super-solutions of
$$-\Delta u\geq V u^p,$$
in $\cX$ or in a proper subset of $\cX$. Such equations have great importance in the theory of local and non-local PDEs. See for instance \cite{BL17, GS81, GS14, GS17, MP99, MP04} and references therein.

We continue to assume (VD), (RVD), (EHI) and $(\widetilde{E}_F)$ in this section.
For the ease of presentation we first consider a non-negative weak super-solution of
\begin{equation}\label{E4.1}
-\Delta u \geq u^p\quad \text{in}\; \cX, \quad \text{and}\quad p \ge 1.
\end{equation}
We say $u\in \cF_{loc}\cap \Cc(\cX)$
is a {\em weak super-solution} of \eqref{E4.1} if for any non-negative $\psi \in\cF\cap\Cc_0(\cX)$ we have
$$\cE(u, \psi) \geq (u^p, \psi).$$

We say that $u$ is {\em superharmonic} if $u \in \mathcal F_{\mbox{\tiny{loc}}}(\cX) \cap C(\cX)$ and $u$ is a weak super-solution to $-\Delta u\geq 0$, i.e., for any non-negative $\psi \in\cF\cap\Cc_0(\cX)$ we have
$$\cE(u, \psi) \geq 0.$$

Our first main result of this section is the following.
\begin{theorem}\label{T4.1}
Grant the hypotheses of Lemma~\ref{L2.4}.
Assume that for a reference point $o\in \cX$ and $p'=\frac{p}{p-1}, p\geq 1,$ we have
$$\liminf_{r\to \infty}\; \max\left\{\frac{F(r)}{\sV(o, r)}, \frac{\sV(o, r)}{(F(r))^{ p'}}\right\}=0\,.$$
Then there is no nontrivial non-negative weak super-solution of \eqref{E4.1} for such $p$.
\end{theorem}

We denote by $U_n=B(o, n)$ and $\uptau_n$ be the exit time from $U_n$. Recall from \cite{FOT} that for any open set
$\Omega$ and any $K\Subset\Omega$ there exists a cut-off function $\varphi\in \cF\cap \Cc_0(\Omega)$ satisfying the following: $0\leq \varphi\leq 1$
and $\varphi=1$ in a neighbourhood of $\overline{K}$. Set of all cut-off function is denote by $(K, \Omega)$.

Let $D, D_1$ be a pair of open set such that $D\Subset D_1\Subset \cX$. Note that, by \cite[Thorem~1.4.2]{FOT} we have $\hat{u}=(m_D^{-1}u)\wedge 1\in \cF$ where $m_D=\sup_{D_1} \abs{u}$. 
Then, since $u$ is non-negative we also have $\hat{u}\varphi\in \cF(D_1)\cap \Cc_0(D_1)$ for $\varphi\in (D, D_1)$. But note that $m_v (\hat{u} \varphi)=u\varphi$ in $D_1$. Therefore $u\varphi\in \cF(D_1)\cap \Cc_0(D_1)$.

Also note that $u\varphi$ is superharmonic in $D$. Take a non-negative $\psi\in \cF(\cX)\cap\Cc_0(D)$. Then due to the local property and \eqref{E4.1} we have
$$\cE(u\varphi, \psi)= \cE(u(\varphi-1), \psi) + \cE(u, \psi)\geq (u^p, \psi)\geq 0.$$
Thus by \cite[Lemma~4.16]{GH08} we have for any $t>0$ that
\begin{equation}\label{needed}
u(x)=u(x)\varphi(x)\geq \Exp_x\left[u(X_t)\Ind_{\{t<\uptau_D\}}\right]=\int_{D}u(y) p_t^D(x, y) d\mu.
\end{equation}
 Again by \cite[Lemma~5.11]{GT12} we know that $p_t^D(x, y)$ is H\"older continuous in $D\setminus\cN$. Hence the RHS of \eqref{needed} is continuous in $D\setminus\cN$ implying that the above inequality holds pointwise in $D\setminus\cN$. We also have the following standard fact.
 
 \begin{lemma}\label{L4.1}
 Fix $x\in D\setminus\cN$. Then $Y_t=u(X_t)\Ind_{\{t<\uptau_D\}}$ is a super-martingale with respect to $\{\sF_{t\wedge\uptau_D}, \Prob_x\}$.
 \end{lemma}
 
 \begin{proof}
 Let $s<t$. Then by strong Markov property
 \begin{align*}
Y_t=\Exp_x\left[u(X_t)\Ind_{\{t<\uptau_D\}}\Big| \sF_{s\wedge\uptau_D}\right] & = \Exp_x\left[ u(X_t)\Ind_{\{s<\uptau_D\}}\Ind_{\{t<\uptau_D\}}
\Big|\sF_{s\wedge\uptau_D}\right]
\\
&= \Ind_{\{s<\uptau_D\}}\Exp_{X_s}\left[u(X_{t-s})\Ind_{\{t-s<\uptau_D\}}\right]
\\
&\leq u(X_s) \Ind_{\{s<\uptau_D\}}=Y_s\,.
\end{align*}
This completes the proof.
 \end{proof}
 
Let $\uuptau$ be any stopping time smaller than $\uptau_D$. Then using Lemma~\ref{L4.1} and optional stopping theorem we find that
$$u(x)\geq \Exp_x\left[u(X_{t\wedge\uuptau})\Ind_{\{\uuptau\wedge t<\uptau_D\}}\right],$$
and letting $t\to\infty$, using Fatou's lamma we finally arrive at
\begin{equation}\label{E4.3}
u(x)\geq \Exp_x\left[u(X_{\uuptau})\Ind_{\{\uuptau<\uptau_D\}}\right].
\end{equation}
Choose $\uuptau_r$ as the hitting time to $\overline{B(o, r)}$ and $\uuptau=\uuptau_r\wedge\uptau_{n}$. Then for any $x\in U_n\setminus \overline{U_r}$ we have from \eqref{E4.3} that
$$u(x)\geq \Exp_x\left[u(X_{\uuptau})\Ind_{\{\uuptau<\uptau_n\}})\right]\geq \Exp_x\left[u(X_{\uuptau})\Ind_{\{\uuptau_r<\uptau_n\}})\Ind_{\{\uuptau_r<\infty\}}\right].$$
Now observe that, denoting $\uptau^*=\lim_{n\to \infty}\uptau_n$, we have $\{\uuptau_r<\infty\}\subset \{\uuptau_r<\uptau^*\}$. This holds because $\uptau^*$ is the explosion time and 
the process can not hit $\overline{B(o, r)}$ after the explosion time. Therefore, letting $n\to\infty$ and using Fatou's lemma we obtain
\begin{equation}\label{E4.4}
u(x)\geq \Exp_x\left[u(X_{\uuptau_r})\Ind_{\{\uuptau_r<\infty\}})\right], \quad x\notin \overline{B(o, r)}\cup\cN.
\end{equation}
Let
$$\sM(r)=\inf_{x\in B(o, r)} u(x).$$
Then we have the following lower bound estimate on $\sM(r)$.

\begin{lemma}\label{L4.2}
Suppose $u \in \mathcal F_{\mbox{\tiny{loc}}}(\cX) \cap C(\cX)$ is a positive superharmonic function, i.e., $-\Delta u\geq 0$. Then the following holds.
\begin{itemize}
\item[(a)] For some constant $\kappa_2$ we have for $r\geq 2$ that
$$\sM(r)\geq \kappa_2 \left( 1\wedge \frac{F(r)}{\sV(o, r)} \right).$$

\item[(b)] There exists a constant $\kappa_3$ satisfying
$$\sM(r)\leq \kappa_3 \sM(2r), \quad r\geq 2.$$
\end{itemize}
\end{lemma}

\begin{proof}
Both proofs use \eqref{E4.4}. Fixing $r=1$ in \eqref{E4.4} and using the continuity of $u$ as well as the continuity of the sample paths, we find that
$$u(x)\geq \left[ \inf_{B(o, 1)} u \right]\Prob_x(\uuptau_1<\infty).$$
For $x\in B^c(o, 2)$ we use Lemma~\ref{L2.6}(a). This proves (a).

Coming to (b). For $x\in B(o, 2r)\setminus (\overline{B(o, r)}\cup\cN)$, we use \eqref{E4.4} and Lemma~\ref{L2.6}(b).
\end{proof}

Now we are ready to prove Theorem~\ref{T4.1}.
\begin{proof}[Proof of Theorem~\ref{T4.1}]
First we assume that $u>0$ in $\cX$.
Applying \cite[Lemma~4.14]{GH08} (for any $\lambda>0$ and then let $\lambda\to 0$) we obtain for $x\in B(o, r)\setminus\cN$ that
$$u(x)\geq \Exp_x\left[\int_0^{\uptau_{B(x, r)}} u^{p}(X_s) ds\right]\geq \sM(2r) \Exp_x[\uptau_{B(x, r)}]\gtrsim \sM^p(r) F(r),$$
by Lemma~\ref{L4.2}(b). Since $x$ is arbitrary, this implies
$$\sM(r)\gtrsim \sM^p(r) F(r) \quad \Rightarrow \sM^{p-1}(r)\lesssim \frac{1}{F(r)}.$$
It is easily seen that $p=1$ is not possible since $F(r)\to \infty$. So we consider $p>1$.
Now apply Lemma~\ref{L4.2}(a) to get
\begin{equation}\label{ET4.1A}
1\wedge \frac{F(r)}{\sV(o, r)}\lesssim \frac{1}{F(r)^{\frac{1}{p-1}}}.
\end{equation}
We pick a sequence of $r$, tending to infinity, satisfying
\begin{equation}\label{ET4.1B}
\lim_{r\to \infty}\; \max\left\{\frac{F(r)}{\sV(o, r)}, \frac{\sV(o, r)}{(F(r))^{ p'}}\right\}=0\,.
\end{equation}
Hence by \eqref{ET4.1A}
$$\sV(o, r)\gtrsim F(r)^{\frac{p}{p-1}}.$$
But this is contradiction to \eqref{ET4.1B}. This, of course, implies there cannot be a solution which is strictly positive in $\cX$.

Now suppose only that $u(x_0)>0$ for some $x_0$. Without loss of generality we assume $x_0=o$ and $\inf_{B(o, \delta)} u>0$. 
Note that by \eqref{needed} and for large enough $D\supset B(o, \delta)$ we have

$$u(x)\geq \Exp_x\left[u(X_t)\Ind_{\{t<\uptau_D\}}\right]\geq \Exp_x\left[u(X_t)\Ind_{B(o, \delta)}\Ind_{\{t<\uptau_D\}}\right].$$
Now let $D\uparrow \cX$ and use Fatou's Lemma to see that 
$$u(x)\geq \left[\inf_{B(o, \delta)} u \right] \Prob_x(X_t\in B(o, \delta)).$$
But the RHS is positive, since there exists a positive heat kernel. This implies $u$ must be positive everywhere. But we have already seen that this is not possible. Hence $u=0$ in $\cX$. This completes the proof.
\end{proof}

Before we proceed further, we should note that the Hunt process or Dirichlet form considered in Theorem~\ref{T4.1} is transient. So we may also ask a similar question for recurrent forms. Recall that a Dirichlet form is 
recurrent if the associated semigroup is recurrent \cite{FOT}. It turns out that one has a stronger result when $\cE$ is recurrent.

\begin{theorem}\label{T4.2}
Let $(\cE, \cF)$ be an irreducible and recurrent Dirichlet form on $\cX$. Then every superharmonic function on $\cX$ must be constant.
\end{theorem}

\begin{proof}
The proof is quite standard. By \cite[Theorem 3.5.6]{CF12} it is known that for any ball $B$ we have
$$\Prob_x(\uuptau_B<\infty)=1\quad \text{q.e. in}\; \cX.$$
Consider a ball $B(o, r)$ around $o\in \cX$. Then $u$ a being super-solution we have (see \eqref{needed}) from above
$$ u(x)\geq \Exp_x[u(X_{\uuptau_B}) \Ind_{\{\uuptau_B<\infty\}}]\geq \inf_{B(o, r)} u\quad \text{q.e. in}\; B^c(o, r).$$
Since $r$ is arbitrary and $u$ is continuous, this implies that $\inf_{\cX} u=u(o)$. But $o$ is also arbitrarily chosen. Hence $u$ must be constant.
\end{proof}

Now in the remaining part of this section we consider non-negative weak super-solutions to
$$-\Delta u\geq V f(u),$$
in $\Omega=\cX\setminus K$ for some compact set $K$ in $\cX$ and $f:[0, \infty)\to[0, \infty)$ is a non-decreasing, continuous function, $f(0)=0$, and
$$\liminf_{t\to 0+} \frac{f(t)}{t^p}>0, \quad p\geq 1.$$
The potential function $V$ is assumed to be continuous and positive. We also assume that $\Omega$ is connected. 

Recall that we say $u\in \cF_{loc}\cap \Cc(\cX)$
is a weak super-solution of 
\begin{equation}\label{E4.7}
-\Delta u\geq V f(u),\quad \text{in}\; \Omega,
\end{equation}
if for any non-negative $\psi\in\cF^0(\Omega)\cap\Cc_0(\Omega)$ we have
$$\cE(u, \psi) \geq (Vf(u), \psi).$$
Before we state our next main result we need few notations and a lemma. Recall that $g^B$ denote the Green function in the ball $B$. Define for $\kappa>0$, $r>1$,
\begin{align}
\Psi(x, r) &=\int_{B(x, r-1)} V(y) g^{B(x, r-1)}(x, y) d\mu(y),\label{E4.8}
\\
\Phi_\kappa(r) &= \inf_{x\in \overline{B(o, r)}\setminus B(o, \frac{r}{2})}\frac{F(r)}{\sV(o, r)}\int_{B(x, \kappa r)} V(y) d\mu(y).\label{E4.9}
\end{align}
where $o$ is a reference point. We also denote $\Phi=\Phi_1$. Then we have the following estimate.

\begin{lemma}\label{L4.3}
There exists a universal $\kappa\in (0, \nicefrac{1}{2})$ such that for all $r$ large we have
\begin{equation}\label{EL4.3A}
\inf_{x\in \overline{B(o, r)}\setminus B(o, \frac{r}{2})} \Psi(x, r)\gtrsim \Phi_\kappa(r).
\end{equation}
\end{lemma}

\begin{proof}
We consider $r>2$. Let $B=B(x,r-1)$. 
Recall from \cite[Theorem~3.12]{GH14} that for some constants $C, K>1$ we have
\begin{equation}\label{EL4.3B}
g^B(x, y)\geq C^{-1}\int_{K^{-1} d(x, y)}^{r-1} \frac{F(s)}{s \sV(x, s)} ds, \quad y\in B(x, K^{-1}(r-1))\setminus \{x\}.
\end{equation}
Therefore, for $x \neq y$ with $d(x, y)\leq K^{-1}(r-1)$ we have from \eqref{EL4.3B} and \eqref{A1} that 
$$g^B(x, y)\geq \frac{1}{C \sV(x, r)} \int_{K^{-1} d(x, y)}^{r-1} \frac{F(s)}{s} ds\gtrsim \frac{F(r-1)-F(K^{-1}\D(x, y)) }{\sV(x, r)}\gtrsim \frac{F(r-1)-F(K^{-1}d(x, y)) }{\sV(o, r)},$$
where the last line follows from \eqref{1.2}. Pick $x\in B(o, r)\setminus B(o, r/2)$. Then for any $\kappa>2$
\begin{align*}
\Psi(x, r) &\geq \frac{1}{\sV(o, r)}\int_{B(x, \frac{r}{K\kappa})} V(y) \left(F(r-1)-F\left(\frac{d(x, y)}{K}\right)\right) d\mu(y).
\\
&\geq \frac{1}{\sV(o, r)}\int_{B(x, \frac{r}{K\kappa})} V(y) \left(F(r-1)-F\left(\frac{r}{K^2\kappa}\right)\right) d\mu(y).
\end{align*}
By \eqref{A1} and \eqref{beta} there exist constants $C_1, C_2>0$ such that
$$F(r-1)-F\left(\frac{r}{K^2\kappa}\right) \geq C_1 F\left(\frac{r}{K^2\kappa}\right)\geq C_2 \frac{1}{K^{2\beta} \kappa^\beta}F(r),$$
and therefore, choosing $\kappa$ large enough we get
$$\frac{1}{2} F(r-1)\geq F\left(\frac{r}{K^2\kappa}\right), \quad \forall\; r>2.$$
Thus 
$$\Psi(x, r)\gtrsim \frac{F(r)}{\sV(o, r)}\int_{B(x, \frac{r}{K\kappa})} V(y) d\mu\gtrsim \Phi_{\frac{1}{K\kappa}}(r).$$
This completes the proof.
\end{proof}

Now we state our next main result.
\begin{theorem}\label{T4.3}
Grant the hypothesis of Lemma~\ref{L2.4} and let $u$ be a non-negative solution of \eqref{E4.7}. Assume that for some $o\in K$ and all $\kappa\in (0, \nicefrac{1}{2})$
$$\lim_{r\to\infty} \, \max\left\{\frac{1}{\Phi_\kappa(r)},  \frac{\sV(o, r)}{F(r)\Phi^{\frac{1}{p-1}}_\kappa(r)}\right\}=0,$$
where $\Phi_\kappa$ is given by \eqref{E4.9}. Then we have $u=0$ in $\Omega$.
\end{theorem}

\begin{proof}
Suppose that $u(x)=0$ at some point $x\in\Omega$ and $B=B(x, 2r)$ is inside $\Omega$.
Then by comparison principle \cite[Lemma~4.14]{GH08} we have 
$$u(z)\geq \int_{B(z, r) } V(y) f(u(y)) g^{B(z, r)} (z, y) d\mu(y),$$
for a.e.~$z\in B(x, r)$. For fixed $r$, it then follows from the calculation of Lemma~\ref{L4.3} that
$$u(z)\gtrsim \int_{B(z, rK^{-1}) } V(y) f(u(y))  d\mu(y)\geq \int_{B(z, rK^{-1})\cap B(x, rK^{-1})} V(y) f(u(y))  d\mu(y).$$
Since $u$, $V$, $f$ are continuous, letting $z\to x$ and using Fatou's lemma, we obtain
$$0=u(x)\gtrsim \int_{B(x, rK^{-1}) } V(y) f(u(y))  d\mu(y).$$
Since $V>0$, the above holds only if $f(u)=0$ or equivalently, $u=0$ in $B(x, rK^{-1})$. Therefore, the set $\{u=0\}$ is both closed and open in $\Omega$.
Since $\Omega$ is connected it follows that $u=0$ everywhere in $\Omega$.

Now we assume that $u>0$ in $\Omega$ and we will derive a contradiction. Take a ball $B(o, l)$ such that $K\subset B(o, l)$ and without loss of generality we may assume 
$l=1/2$. Note that it suffices to show that $u=0$ in $B^c(o, 1)$. 
We claim that 
$$\inf_{B^c(o, 1)} u=0.$$
If not, let us assume that $\inf_{B^c(o, 1)} u=\delta>0$.
Then for any $x\in \Omega\setminus\cN$ with $r=d(o,x)>2$, we have
$$u(x)\geq f(\delta) \Exp_x\left[\int_0^{\uptau_{B(x, r-1)}} V(X_s) ds\right]= f(\delta)\Psi(x, r),$$
where $\Psi$ is given by \eqref{E4.8}. Using Lemma~\ref{L4.3} we get
$$u(x)\gtrsim f(\delta) \Phi_\kappa(r)\,. $$
Since $\Phi_\kappa(r)\to \infty$, as $r\to\infty$, the above implies
 that $\lim_{\D(o, x)\to \infty} u(x)=\infty$. On the other hand, the Hunt process is transient. Therefore, there exists 
$x\in B^c(o, 1)\setminus \cN$ such that $\Prob_{x}(\uuptau_{B(o, 1)}<\infty)<1$. Otherwise, the process would hit $B(0, 1)$ infinitely often and would also
exist for all time. But this would be contradicting to \cite[Theorem~3.5.2]{CF12} which says the transient process should go to infinity if exists for all time.
Pick a point $x\in B^c(o, 1)\setminus \cN$  with such property. Let $r>> 1$. By superharmonicity, we have, for this $x$,
$$u(x)\geq \Exp_x[u(X_{\uuptau_{B_1}\wedge\uptau_n})]\geq
\Exp_x[u(X_{\uuptau_{B_1}}) \Ind_{\{\uuptau_{B_1}<\uptau_n\}}] + \left[\min_{\partial B(o, n)} u\right]\, \Prob_x(\uuptau_{B_1}>\uptau_n),$$
where $B_1 = B(o,1)$.
Now let $n\to\infty$ to get a contradiction. Hence we get $\inf_{B^c(o, 1)} u=0$.

Let us now define
$$\sM(r)=\inf_{x\in B(o, r)\cap B^c(o, 1)} u(x).$$
It is easy to see that Lemma~\ref{L4.2} holds.
We choose a sequence $r_n\to\infty$ satisfying the following: there exists $x_n$ with $d(o,x_n)=r_n$ and $u(x_n)=\sM(r_n)\leq \frac{1}{2} \sM(r_{n-1})$.
This is possible to do since $\lim_{r\to\infty} \sM(r)=0$, as we have shown above. 
Now we proceed as in Theorem~\ref{T4.1}. For each $n$ we have
\begin{align*}
\sM(r_n)=u(x_n) &\geq \Exp_x\left[\int_0^{\uptau_{B(x_n, r_n-1)}} V(X_s) f(u(X_s)) ds\right]
\\
&\geq \Exp_x\left[\int_0^{\uptau_{B(x_n, r_n-1)}} V(X_s) f(\sM(2r_n-1)) ds\right]
\\
&\gtrsim \sM^p(2r_n-1) \int_{B(x_n, r_n-1)} V(y) g^{B(x_n, r_n-1)}(x_n, y) d\mu(y)
\\
&\gtrsim \sM^p(r_n-1/2) \Psi(x_n, r_n)
\\
&\gtrsim \sM^p(r_n)\Phi_\kappa(r_n),
\end{align*}
by \eqref{EL4.3A}. Now applying Lemma~\ref{L4.2} we arrive at
$$1\gtrsim \sM^{p-1}(r_n) \Phi_\kappa(r_n)\gtrsim \left[ 1\wedge \frac{F(r_n)}{\sV(o, r_n)}\right]^{p-1} \Phi_\kappa(r_n).$$
Since $\Phi_\kappa(r_n)\to\infty$, it is evident that $p>1$. It also follows that $\frac{F(r_n)}{\sV(o, r_n)}\to 0$ as $r_n\to\infty$. Hence we must have
$$1\lesssim \frac{\sV(o, r_n)}{F(r_n)\Phi_\kappa^{\frac{1}{p-1}}(r_n)}\to 0,$$
as $r_n\to \infty$. Thus we have a contradiction. 

Therefore, $u=0$ is the only possible solution.
\end{proof}

Below we compare Theorem~\ref{T4.3} with some existing works in literature. 
\begin{itemize}
\item Let us compare Theorem~\ref{T4.3} with \cite{GS17} where similar problem has been studied on Riemannian manifolds and $f(t)=t^p$. Suppose
$$\sV(o, r)\simeq r^\alpha, \quad G(x, y)\simeq (d(x, y))^{-\gamma},$$
for some positive $\alpha, \gamma$ and $r, d(x, y)>R$. Using Lemma~\ref{L2.1} we also have $F(r)\gtrsim r^{\alpha-\gamma}$. Also suppose
for some $m>\gamma-\alpha$ it holds that $V(x)\geq c d(o, x)^m$. In view of Theorem~\ref{T4.3} let us compute $\Phi_\kappa$. It follows that for any $\kappa\in (0, \nicefrac{1}{2})$
\begin{align*}
\Phi_\kappa(r)\gtrsim r^{-\gamma} \inf_{x\in \overline{B(o, r)}\setminus B(o, r/2)}\int_{B(x, \kappa r)} d(o, y)^m d\mu(y).
\end{align*}
Since $d(o, x)\geq \frac{r}{2}$ then $d(o, y)\simeq r^m$ in $B(x, \kappa r)$. Combining we have
$$\Phi_\kappa(r)\gtrsim r^{\alpha+m-\gamma}.$$
Since $m>\gamma-\alpha$, we get $\Phi_\kappa(r)\to\infty$ as $r\to\infty$. Again if we consider $1<p<\frac{\alpha+m}{\gamma}$ i.e., 
$\frac{1}{p-1}>\frac{\gamma}{\alpha+m-\gamma}$, then we also have
$$\lim_{r\to\infty} \frac{\sV(o, r)}{F(r)\Phi_\kappa^{\frac{1}{p-1}}(r)}=0.$$
Therefore, \eqref{E4.7} does not have any non-trivial non-negative super-solution with exponent $p$. It should be observed that the above critical
exponent $\frac{\alpha+m}{\gamma}$ matches with \cite{GS17}.

\item  \cite{MP99, MP04} consider problems  of type \eqref{E4.7} with $f(t)=t^p$ on $\Rd$ and come up with a set of conditions for non-existence of positive solutions.
In our setting, an analogous condition would be, for $p>1$,
\begin{equation}\label{E4.12}
\limsup_{r\to\infty} \frac{1}{F(r)^{\frac{p}{p-1}}}\int_{B(o, r)\setminus B(o, r/2)} V^{-\frac{1}{p-1}} d\mu(y)=\ell\in [0, \infty).
\end{equation}
Now suppose that $\ell=0$ and we also have for any $\kappa\in (0, \nicefrac{1}{2})$
$$\liminf_{r\to\infty}\inf_{x\in \overline{B(o, r)}\setminus B(o, r/2)}\frac{\mu\left((B(o, r)\setminus B(o, r/2))\cap B(x,  \kappa r)\right)}{ \sV(o, r) }>0.$$
Then we note that, for $A(r, x)= (B(o, r)\setminus B(o, r/2))\cap B(x, \kappa r)$, we have from \eqref{E4.12} 
\begin{align*}
\int_{B(x,  \kappa r)} V(y) d\mu(y)\geq \int_{A(x, r)} V(y) d\mu(y) & \geq \mu(A(x, r)) \left[ \fint_{A(x, r)} V(y)^{-\frac{1}{p-1}} d\mu(y)\right]^{-(p-1)}
\\
& \geq \mu(A(x, r))^p \left[ \int_{B(o, r)\setminus B(o, r/2)} V(y)^{-\frac{1}{p-1}} d\mu(y)\right]^{-(p-1)}
\end{align*}
Hence 
\begin{align*}
\Phi_\kappa(r)&\geq \frac{F(r)}{\sV(o, r)} \mu(A(x, r))^p \left[ \int_{B(o, r)\setminus B(o, r/2)} V(y)^{-\frac{1}{p-1}} d\mu(y)\right]^{-(p-1)}
\\
&\gtrsim \frac{F(r)}{\sV(o, r)} \sV(0, r)^p \ell(r)^{-1} F(r)^{-p}
\\
&= \ell(r)^{-1} \frac{\sV(o, r)^{p-1}}{F(r)^{p-1}},
\end{align*}
where $\ell(r)\to 0$, as $r\to\infty$. This of course, implies
$$\lim_{r\to\infty} \frac{\sV(o, r)}{F(r)\Phi_\kappa^{\frac{1}{p-1}}(r)}=0.$$
This gives us the required condition of Theorem~\ref{T4.3} and therefore, there can not be any non-trivial positive super-solution.
\end{itemize}

\begin{remark}
The probabilistic approach we have employed here can be used to address similar problems for system of equations. For example, Lane-Emden type of systems \cite{DF94, S95}. The proof
would be fairly routine and can be deduced from the proof of Theorem~\ref{T4.3}. For more details, we refer to  \cite{B18}, \cite[Theorem~5.4]{BL17}.
\end{remark}

\section{Local Faber-Krahn inequality}\label{Sec-FK}

In this section we prove a local Faber-Krahn inequality on local or non-local regular Dirichlet spaces. Let $(\cX,d)$ be a locally compact separable metric space, $\mu$ a locally finite Radon measure on $\cX$ with full support, and $(\cE,\cF)$ a regular Dirichlet form on $L^2(\cX,\mu)$. We assume that all metric balls in $(\cX,d)$ are relatively compact.
In addition to (VD) and (RVD), we assume that
\begin{align} \label{V(x,1)}
c \le \inf_{x \in \cX} \sV(x,1) \le \sup_{x \in \cX} \sV(x,1) \le C
\end{align}
for some positive constants $c,C$,
and that the mean exit time function is given by an increasing bijection $F:[0,\infty) \to [0,\infty)$ satisfying
\begin{align} \label{F beta}
F(t) \simeq t^{\beta}, \quad t>0,
\end{align}
where $\alpha_2 > \beta$. That is,
\begin{align} \label{mean exit time}
\Exp_x \tau_{B(x,r)} \simeq F(t) \simeq t^{\beta}, \quad x \in \cX, r>0.
\end{align}
Further, we assume that $(\cE,\cF)$ admits a heat kernel $p(t,x,y)$ that satisfies the upper bound
\begin{align} \label{eq:HK upper bound}
p(t,x,y) \le \frac{C_1}{\sV(x,\mathcal R(t))} H \left( \frac{F(d(x,y))}{t} \right),
\end{align}
where $\mathcal R$ is the inverse of $F$, and $H:\RR \to [0,\infty)$ is non-increasing and satisfies
\begin{align} \label{eq:H}
\sum_{k=1}^{\infty} k^{\beta}  \int_{c k^{\beta}}^{\infty}  t^{\alpha_2/\beta-2} H(t)dt < \infty
\end{align}
for some constant $c \sim 1$.

In particular, the upper bound \eqref{eq:HK upper bound} holds in the following special cases:
\begin{enumerate}
\item
the stable-like estimate
\begin{align} \label{eq:upper HK nonlocal alpha}
p(t,x,y) \le \frac{C_1}{t^{\alpha/\beta}} \left( 1+ \frac{d(x,y)}{c_2 t^{1/\beta}} \right)^{-(\alpha+\beta)},
\end{align}
on an $\alpha$-regular space with $F(t) = t^{\beta}$ and $\beta >1$,
\item
the sub-Gaussian estimate
\begin{align}\label{eq:upper HK local alpha}
p(t,x,y) \le \frac{C_1}{t^{\alpha/\beta}} \exp\left(-c \left(\frac{\D(x, y)^{\beta}}{t} \right)^{\frac{1}{\beta-1}}\right),
\end{align}
on an $\alpha$-regular space with $F(t) = t^{\beta}$ and $\beta >1$,
\item
the sub-Gaussian estimate
\begin{align}\label{eq:upper HK local}
p(t,x,y) \le \frac{C_1}{\sV(x,\mathcal R(t))} \exp \left( -\frac{t}{2} \Phi \left( c \frac{d(x,y)}{t} \right) \right),
\end{align}
where $\Phi$ is as in \eqref{Phi} and $F$ satisfies \eqref{F beta}.
\end{enumerate}

\begin{definition} \label{def:median exit time} 
We define the {\em median exit time} for the diffusion starting at point $o \in \Omega$ as
\begin{align*}
T_{\eta}(o) := \inf \left\{ t>0 : \Prob_{o}(\tau \le t) \ge \eta \right\}.
\end{align*}
\end{definition}
The implicit constant in Theorem \ref{thm:FK} depends on the choice of $\eta$. We make the arbitrary choice $\eta=1/2$ and drop the subscript $\eta$. 


\begin{theorem} \label{thm:FK}
Assume that (VD), (RVD), \eqref{V(x,1)}, \eqref{mean exit time}, and the heat kernel upper bound \eqref{eq:HK upper bound} hold. 
Let $\Omega \subset \cX$ be a bounded domain. 
Suppose $V:\cX \to \RR$ is bounded and Borel measurable.
If $u \in \cF^0(\Omega)$ is a non-trivial weak solution of 
\begin{align}
\begin{split}
 \cE(u,\phi) + \int_{\Omega} Vu \phi \, d\mu &= 0 ,\quad \forall \phi \in \cF^0(\Omega).
\end{split}
\end{align}
and if $o$ is a point in $\Omega \setminus \cN$ such that $|u(o)| \ge \frac{3}{4} \|u\|_{\infty}$ then there exists a ball $B \subset \cX$ of radius $\sR(T(o))$ such that 
$$ \| V^- \|_{L^{\frac{\alpha_1}{\alpha_1 - \alpha_2 + \beta},1}(\Omega \cap B)} \geq c,$$
where $c$ is a positive constant that depends only on $\alpha_1$, $\alpha_2$, $\beta$, on the constants in (VD), (RVD), \eqref{V(x,1)} and \eqref{F beta}, on the constants and the function $H$ in the heat kernel upper bound.
\end{theorem}

%
%

We consider the solution $u$ of \eqref{eq:equation} as a steady-state solution of the parabolic equation
$$ -\int_0^t \int u \frac{\partial}{\partial t} \phi \, d\mu \, ds + \int_0^t \cE(u,\phi) ds + \int_0^t \int V u \phi \, d\mu \, ds = 0.$$
Let $(X_t)_{t \ge 0}$ be the diffusion process on $\Omega$ with absorption at the boundary of $\Omega$, associated with the Dirichlet form $(\cE,\cF^0(\Omega))$.
By the Feynman-Kac formula, 
$$ u(x) =  \Exp_{x} \left(u(X_t) \exp \left( -\int_0^t V(X_s) ds \right)  1_{\left\{t<\tau\right\}} \right), \quad \forall x \in \Omega \setminus \cN, t > 0.$$
Let $o$ be a point in $\Omega \setminus \cN$ such that $|u(o)| \ge \frac{3}{4} \|u\|_{\infty}$.  We may assume that $u(o) > 0$ (otherwise consider $-u$ and note that $-u$ also solves \eqref{eq:equation}). Let $\tau=\tau_{\Omega}$ be the first exit time from $\Omega$, i.e.~the time the process gets absorbed at the boundary. Then,
\begin{align*}
 u(o)
&= \Exp_{o} \left(u(X_t) 1_{\left\{t<\tau\right\}}\exp \left( -\int_0^t V(X_s) ds \right)  \right) \\
 &\leq \|u\|_{\infty}  \Exp_{o}\left( 1_{\left\{t < \tau \right\}}  \exp \left( \int_0^t V^-(X_s) ds \right) \right).
 \end{align*}
Since $u(o) \ge \frac{3}{4} \|u\|_{\infty}$, this simplifies to
\begin{align} \label{eq:fund ineq}
  \Exp_{o}\left(   1_{\left\{\tau > t\right\}}  \exp \left( \int_0^t V^-(X_s) ds \right)\right) \geq \frac{3}{4}.
\end{align}

\begin{lemma}[Khasminskii's lemma] \label{lem:Khasminskii} 
Let $(X_s)_{s\ge0}$ be a Markov process on the metric measure space $(\cX,d,\mu)$ and let $V:\cX \to [0,\infty)$ be a Borel measurable function. If, for some $t>0$ and $c < 1$,
$$ \sup_{x \in X}{ \Exp_{x} \left[ \int_0^t{V(X_s)} ds \right] }= c,$$
then
$$ \sup_{x \in X}{ \Exp_{x} \left[ \exp \left( \int_{0}^{t}{V(X_s) ds} \right) \right]} \leq \frac{1}{1-c}.$$
\end{lemma}

The proof of Khasminskii's lemma is given in \cite{feyn, sim} for the case $\cX = \RR^n$ and extends verbatim to the general case.

\begin{lemma}\label{lem:est} Let $f: \cX \rightarrow [0,\infty)$. For every $0 < d \lesssim 1$,
$$ \sup_{x \in X} \int_{X}{ f(y) \int_{0}^{F(d)}{p_s(x,y) ds } \, dy} \lesssim \sup_{x \in X} \| f\|_{L^{\frac{\alpha_1}{\alpha_1 - \alpha_2 + \beta},1}(B(x,d))}.$$
\end{lemma}

\begin{proof} 
We fix $x \in X$ and let $B_1 = B(x,d)$. We choose countably many balls $B_i=B(x_i,d)$, $i \ge 2$, in such a way that the balls $B(x_i,d/5)$ are disjoint, and every point in $\cX$ is contained in at most $N$ of the balls $B_i$, where $N$ is finite and depends only on the volume doubling constant.
Then
$$ \int_X{ f(y) \int_{0}^{F(d)}{p_s(x,y) ds} \, dy} \leq \sum_{i=1}^{\infty} \int_{B_i}{ f(y) \int_{0}^{F(d)}{p_s(x,y) ds} \, dy }.$$

We apply the upper heat kernel bound \eqref{eq:HK upper bound}, the substitution $t = \frac{F(d(x,y))}{s}$, $ds = - t^{-2} F(d(x,y))dt$, the volume doubling property \eqref{1.2}, and \eqref{F beta} to obtain, for each of the balls $B_i$,
\begin{align*}
 \int_{B_i} f(y) \int_{0}^{F(d)} p_s(x,y) ds \, dy 
& \lesssim  \int_{B_i} f(y) \int_{0}^{F(d)} \frac{1}{\sV(x,\sR(s))} H\left(\frac{F(d(x,y))}{s}\right) ds \, dy  \\
& =  \int_{B_i} f(y) \int_{\frac{F(d(x,y))}{F(d)}}^{\infty} \frac{1}{\sV(x,\sR(F(d(x,y))/t))} H(t) t^{-2} F(d(x,y)) dt \, dy  \\
& \lesssim \frac{1}{\sV(x,d)}  \int_{B_i} f(y) \int_{\frac{F(d(x,y))}{F(d)}}^{\infty} \left(\frac{\sR(F(d))}{\sR(F(d(x,y))/t)} \right)^{\alpha_2} H(t) t^{-2} F(d(x,y)) dt \, dy \\
& \lesssim  \frac{1}{\sV(x,d)} \int_{B_i} f(y) \int_{\frac{F(d(x,y))}{F(d)}}^{\infty} \left(\frac{F(d)}{F(d(x,y))/t} \right)^{\alpha_2/\beta} H(t) t^{-2} F(d(x,y)) dt \, dy \\
& =  \frac{F(d)^{\alpha_2/\beta}}{\sV(x,d)} \int_{B_i} f(y) F(d(x,y))^{-\alpha_2/\beta + 1} \int_{\frac{F(d(x,y))}{F(d)}}^{\infty}  t^{\alpha_2/\beta-2} H(t)  dt \, dy \\
& \lesssim  \int_{B_i} f(y) F(d(x,y))^{-\alpha_2/\beta + 1} \int_{\frac{F(d(x,y))}{F(d)}}^{\infty}  t^{\alpha_2/\beta-2} H(t)  dt \, dy.
\end{align*}
In the last line we used the assumption $F(t) \simeq t^{\beta}$ and the fact that $\sV(x,d)^{-1} \lesssim d^{-\alpha_2} \sV(x,1)^{-1}$ which follows from $d \lesssim 1$, (VD) and (RVD), and we also used the lower bound \eqref{V(x,1)}.

If $F(d(x,y)) \ge F(d)$ then, by \eqref{eq:H},
\begin{align*}
\int_{\frac{F(d(x,y))}{F(d)}}^{\infty}  t^{\alpha_2/\beta-2} H(t)  dt
& \lesssim  \int_1^{\infty}  t^{\alpha_2/\beta-2} H(t)  dt \lesssim 1.
\end{align*}
If $F(d(x,y)) < F(d)$ then
\begin{align*}
\int_{\frac{F(d(x,y))}{F(d)}}^{\infty}  t^{\alpha_2/\beta-2} H(t)  dt
& \lesssim \int_{\frac{F(d(x,y))}{F(d)}}^1  t^{\alpha_2/\beta-2} H(t)  dt
+ \int_1^{\infty}  t^{\alpha_2/\beta-2} H(t)  dt \\
& \lesssim H(0) \int_{\frac{F(d(x,y))}{F(d)}}^1  t^{\alpha_2/\beta-2}  dt
+ 1 \\
& \lesssim H(0) \frac{\beta}{\alpha_2-\beta} \left[ 1 - \left( \frac{ F(d(x,y))}{F(d)} \right)^{\frac{\alpha_2}{\beta}-1} \right] + 1 \lesssim 1.
\end{align*}
Hence, in any case,
\begin{align} \label{eq:int f Psi}
 \int_{B_i} f(y) \int_{0}^{F(d)} p_s(x,y) ds \, dy 
& \lesssim 
\int_{B_i} f(y)  F(d(x,y))^{-\alpha_2/\beta + 1} dy.
\end{align}

For each ball $B_i$ that is at most distance $d/2$ away from $x$, we estimate the right hand side as follows. The number of such $B_i$'s can be bounded in terms of the volume doubling constant only.
We apply the Lorentz-H\"older inequality due to O'Neil \cite[Theorem 3.5]{oneil},
 $$ \| fg \|_{L^1} \lesssim \|f\|_{L^{\frac{\alpha_1}{
\alpha_1 - \alpha_2 + \beta}, 1}} \| g\|_{L^{\frac{\alpha_1}{\alpha_2-\beta}, \infty}},$$
to the right hand side of \eqref{eq:int f Psi} to get
\begin{align*}
 \int_{B_i} f(y) \int_{0}^{F(d)} p_s(x,y) ds \, dy 
& \lesssim  \| f \|_{L^{\frac{\alpha_1}{\alpha_1 - \alpha_2 + \beta},1}(B_i)} \left\| F(d(x,y))^{-\alpha_2/\beta +1}  \right\|_{L^{\frac{\alpha_1}{\alpha_2-\beta},\infty}(B_i)}.
\end{align*}
We claim that $\left\| F(d(x,y))^{-\alpha_2/\beta +1}  \right\|_{L^{\frac{\alpha_1}{\alpha_2-\beta},\infty}(B_i)} < C$.
Indeed, 
by \eqref{F beta} we have 
\begin{align*}
 F(d(x,y))^{-\alpha_2/\beta +1} 
&  \lesssim d(x,y)^{-\alpha_2 + \beta}.
\end{align*}
We then estimate using \eqref{1.2},
\begin{align*}
 \left\| d(x,y)^{-\alpha_2 + \beta}  \right\|^{\frac{\alpha_1}{\alpha_2-\beta}}_{L^{\frac{\alpha_1}{\alpha_2-\beta},\infty}(B(x,3d/2))} 
& = \sup_{t>0} t^{\frac{\alpha_1}{\alpha_2-\beta}} \mu \left( y: d(x,y) < 3d/2, d(x,y)^{-\alpha_2 + \beta} > t \right) \\
& = \sup_{t>0} t^{\frac{\alpha_1}{\alpha_2-\beta}} \mu \left( y: d(x,y) < 3d/2 \wedge  t^{\frac{1}{-\alpha_2 + \beta}} \right) \\
& = \sup_{t>0} t^{\frac{\alpha_1}{\alpha_2-\beta}} \mu\left(B\left(x,3d/2 \wedge  t^{\frac{1}{-\alpha_2 + \beta}} \right)\right) \\
& \lesssim \sup_{t>0} t^{\frac{\alpha_1}{\alpha_2-\beta}}  \left( 1 \wedge \left( \frac{2}{3d} t^{\frac{1}{-\alpha_2 + \beta}} \right)^{\alpha_1}  \right) \mu(B(x,3d/2)) \\
& = \sup_{t>0} \left( t^{\frac{\alpha_1}{\alpha_2-\beta}}  \wedge \left( \frac{2}{3d} \right)^{\alpha_1}  \right) \sV(x,3d/2) \\
& \lesssim ( 1 \wedge d^{-\alpha_1}) \sV(x,3d/2).
\end{align*}
The right hand side is $\lesssim 1$ because $d \lesssim 1$ and by (VD) and (RVD).

Now we have proved that
\begin{align*}
 \int_{B_i} f(y) \int_{0}^{F(d)} p_s(x,y) ds \, dy 
& \lesssim  \| f \|_{L^{\frac{\alpha_1}{\alpha_1 - \alpha_2 + \beta},1}(B_i)}
\end{align*}
for any of the balls $B_i$ that is at most distance $d/2$ away from $x$.

It  remains to consider those balls $B_i$ whose distance to $x$ is at least $d/2$. 
By the volume doubling property, the number $N_k$ of balls $B_j$ with distance to $x$ being $\sim kd$ is at most $\simeq k^{\alpha_2}$. Indeed, 
$$ N_k \min_j \sV(x_j,d/5) \leq \mu\left(\cup_j B(x_j,d/5) \right) \leq \sV(x,(k+1)d) \lesssim k^{\alpha_2} \min_j \sV(x_j,d/5).$$

Consider a ball $B_i$ that is at distance $\sim k d$ from $x$. For any $y \in B_i$, by \eqref{F beta},
$$ F(d(x,y))^{-\alpha_2/\beta +1} \lesssim F(d)^{-\alpha_2/\beta+1} k^{-\alpha_2 + \beta}. $$
We refine the estimate \eqref{eq:int f Psi} as follows.
\begin{align}
 \int_{B_i} f(y) \int_{0}^{F(d)} p_s(x,y) ds \, dy 
& \lesssim  \int_{B_i} f(y) F(d(x,y))^{-\alpha_2/\beta + 1} \int_{\frac{F(d(x,y))}{F(d)}}^{\infty}  t^{\alpha_2/\beta-2} H(t)  dt \, dy \\
& \lesssim F(d)^{-\alpha_2/\beta+1} k^{-\alpha_2 + \beta} \int_{B_i} f(y) \int_{c k^{\beta}}^{\infty}  t^{\alpha_2/\beta-2} H(t)  dt \, dy,
\end{align}
for some constant $c \sim 1$.

Applying the H\"older inequality, \eqref{F beta}, (RVD) and \eqref{V(x,1)}, we get
\begin{align*} 
 \int_{B_i} f(y) \int_{0}^{F(d)} p_s(x,y) ds \, dy 
& \lesssim F(d)^{-\alpha_2/\beta+1} k^{-\alpha_2 + \beta} \int_{c k^{\beta}}^{\infty}  t^{\alpha_2/\beta-2} H(t) dt \, \| f \|_{L^{\frac{\alpha_1}{\alpha_1 - \alpha_2 + \beta}}(B_i)} \left\| 1 \right\|_{L^{\frac{\alpha_1}{\alpha_2-\beta}}(B_i)} \\
& \lesssim  d^{-\alpha_2 + \beta} k^{-\alpha_2 + \beta} \int_{c k^{\beta}}^{\infty}  t^{\alpha_2/\beta-2} H(t)  dt \, \| f \|_{L^{\frac{\alpha_1}{\alpha_1 - \alpha_2+\beta}}(B_i)} \left(\frac{V(x,d)}{V(x,1)}\right)^{\frac{\alpha_2-\beta}{\alpha_1}} V(x,1)^{\frac{\alpha_2-\beta}{\alpha_1}} \\
& \lesssim  k^{-\alpha_2 + \beta} \int_{c k^{\beta}}^{\infty}  t^{\alpha_2/\beta-2} H(t)  dt \, \| f \|_{L^{\frac{\alpha_1}{\alpha_1 - \alpha_2+\beta}}(B_i)}.
\end{align*}

Altogether, we obtain
\begin{align*}
 \sup_{x \in X} \int_{X}{ f(y) \int_{0}^{F(d)}{p_s(x,y) ds} \, dy}
 &\lesssim \left( 1 + \sum_{k=1}^{\infty}  k^{\beta}  \int_{c k^{\beta}}^{\infty}  t^{\alpha_2/\beta-2} H(t)dt \right)  \sup_{x \in X} \| f \|_{L^{\frac{\alpha_1}{\alpha_1 - \alpha_2 + \beta},1}(B(x,d))}.
\end{align*}
By assumption \eqref{eq:H}, the sum is bounded.
\end{proof}

Given a point $o$ where the solution attains its maximum, the estimate \eqref{eq:fund ineq} together with the Cauchy-Schwarz inequality imply that
\begin{align*}
 \frac{3}{4} \leq  \Exp_{o}\left(   1_{\left\{\tau > t\right\}}  \exp \left( \int_0^t V^-(X_s) ds \right) \right) \leq \mathbb{P}\left( \tau > t \right)^{1/2} \left(\mathbb{E}_{o} \exp \left( \int_0^t 2V^-(X_s) ds \right)\right)^{1/2}
\end{align*}
for all $t > 0$.
We choose $t = T(o)$ to be the median exit time $T(o)$ that we introduced in Definition \ref{def:median exit time}. Then
$$ \mathbb{P}_{o}\left( \tau > T(o) \right)^{1/2} \le \frac{1}{\sqrt{2}},$$
and therefore
$$   \mathbb{E}_{o} \left( \exp \left( \int_0^{T(o)} 2V^-(X_s) ds \right) \right) \geq \frac{9}{8}.$$
Khasminskii's Lemma (Lemma \ref{lem:Khasminskii}) implies
$$  \sup_{x \in X}{ \Exp_{x} \left( \int_0^{T(o)}{ 2V^-(X_s)} ds \right) } \geq \frac{1}{9}.$$
We would like to apply Lemma \ref{lem:est} with $f = 2V^+$ and $d = \sR(T(o))$. To this end we need to verify that $d \lesssim 1$. An easy application of Chebyshev's inequality gives $T(o) \le 2 \Exp_{o}\tau_{\Omega}$ and since $\Exp_{o}\tau_{\Omega} \le \Exp_{o}\tau_{B(o,\mbox{diam}(\Omega))} \lesssim F(\mbox{diam}(\Omega))$ by \eqref{F beta}, we indeed have $d \lesssim 1$ provided that $\Omega$ is bounded.
Now Lemma \ref{lem:est} gives
$$ \frac{1}{9} \leq \sup_{x \in X} \Exp_{x} \left( \int_0^{T(o)}{2V^-(X_s)} ds \right)   \lesssim   \sup_{x \in X} \left\| V^-\right\|_{L^{\frac{\alpha_1}{\alpha_1 - \alpha_2 + \beta},1}(B(x,\sR(T(o))))}.
$$
Recall that $V \equiv 0 $ outside $\Omega$. The proof of Theorem \ref{thm:FK} is now complete.

\subsection*{Acknowledgements}
This research of Anup Biswas was supported in part by an INSPIRE faculty fellowship and DST-SERB grant EMR/2016/004810. This work was supported by a grant from the Simons Foundation/SFARI (585415, JL).




\begin{thebibliography}{10}
\bibitem{Aik98} H. Aikawa, 
Norm estimate of Green operator, perturbation of Green function and integrability of superharmonic functions. Math. Ann. 312 (1998), no. 2, 289--318. 

\bibitem{BBCK} M. T. Barlow, R. F. Bass, Z-Q Chen, and M. Kassmann, Non-local Dirichlet forms and symmetric jump processes. Trans. Amer. Math. Soc. 361 (2009)
1963--1999.

\bibitem{BB99} M.T. Barlow and R.F. Bass, Brownian motion and harmonic analysis on Sierp\'{i}nski carpets, Canad.
J. Math. (4) 51 (1999), 673--744.

\bibitem{BP} M.T. Barlow and E.A. Perkins, Brownian motion on the Sierpi\'nski gasket, Probab. Theory. Related
Fields 79 (1988), 543--623.

\bibitem{B17} A. Biswas, Location of maximizers of eigenfunctions of fractional Schr\"{o}edinger's equation, Mathematical
Physics, Analysis and Geometry 20 (2017) no. 4,  pages 14

\bibitem{B18} A. Biswas, Liouville type results for system of equations involving fractional Laplacian in the exterior domain, Preprint. arXiv:1810.03265, 2018

\bibitem{BL17}  A. Biswas and J. L\"{o}rinczi, Maximum principles and Aleksandrov-Bakelman-Pucci type estimates for non-local Schr\"{o}dinger equations with exterior conditions, Preprint. arXiv:1710.11596, 2017

\bibitem{DCH} L. De Carli and S. M. Hudson, A Faber-Krahn inequality for solutions of Schr\"{o}dinger's equation. Adv. Math., 230(4-6) (2012) 2416--2427,

\bibitem{DEH} L. De Carli, J. Edward, S. Hudson and M. Leckband, Minimal support results for Schr\"odinger equations.
Forum Math. 27 (2015), no. 1, 343--371

\bibitem{DF94} D.G. De Figueiredo and P.L. Felmer, A Liouville-type theorem for elliptic systems, Ann. Sc. Norm. Super. Pisa Cl. Sci. (4) XXI (1994) 387--397.

\bibitem{CF12}  Z-Q Chen and M. Fukushima, 
Symmetric Markov processes, time change, and boundary theory. London Mathematical Society Monographs Series, 35. Princeton University Press, Princeton, NJ, 2012.

\bibitem{DEL}  J. Dolbeault, J. M. Esteban and A. Laptev, Spectral estimates on the sphere. Anal. PDE 7 (2014), no. 2, 435--460.

\bibitem{FOT} M. Fukushima, Y. Oshima and M. Takeda, Dirichlet forms and symmetric Markov processes, Second
revised and extended edition, De Gruyter, Studies in Mathematics, 19, 2011.

\bibitem{GM16} B. Georgiev and M. Mukherjee. Nodal geometry, heat diffusion and Brownian
motion. Analysis and PDE, 11 (2018), no. 1, 133--148

\bibitem{GS81} B. Gidas and J. Spruck, Global and local behavior of positive solutions of nonlinear elliptic equations,
Comm. Pure Appl. Math., 34 (1981) 525--598.

\bibitem{G80} B. Gidas. Symmetry properties and isolated singularities of positive solutions of nonlinear elliptic equations. In Non-
linear partial differential equations in engineering and applied science (Proc. Conf., Univ. Rhode Island, Kingston,
R.I., 1979), volume 54 of Lecture Notes in Pure and Appl. Math., pages 255--273. Dekker, New York, 1980.

\bibitem{GGK}  I. Gohberg, S. Goldberg, and M. A. Kaashoek. Classes of linear operators. Vol. I. Operator Theory: Advances and Applications, 49. Birkh\"auser Verlag, Basel

\bibitem{GKMR} R. Gorenflo, A.A. Kilbas, F. Mainardi and S.V. Rogosin: Mittag-Leffler Functions, Related Topics and Applications, Springer, 2014

\bibitem{GH08} A. Grigor'yan and J. Hu, Off-diagonal upper estimates for the heat kernel of the Dirichlet forms on metric spaces, Invent. Math. 174 (2008), 81--126.

\bibitem{GT12} A. Grigor'yan and A. Telcs, Two-sided estimates of heat kernels on metric measure spaces. Ann. Probab. 40 (2012), no. 3, 1212--1284.

\bibitem{GH09} A. Grigor'yan and J. Hu,  Upper bounds of heat kernels on doubling spaces. Mosc. Math. J. 14 (2014), no. 3, 505--563, 641--642.

\bibitem{GH14} A. Grigor'yan and J. Hu, Heat kernels and Green functions on metric measure spaces. Canad. J. Math. 66 (2014), no. 3, 641-- 699.

\bibitem{GHL09} A. Grigor'yan, J. Hu and Ka-Sing Lau, Heat kernels on metric spaces with doubling measure,
In: Bandt C., Z\"ahle M., M\"orters P. (eds) Fractal Geometry and Stochastics IV. Progress in Probability, vol 61. Birkh\"auser Basel, 2009

\bibitem{GS14} A. Grigor'yan, Y. Sun, On non-negative of the inequality $\Delta u + u^\sigma \leq 0$ on Riemannian manifolds,
Comm. Pure Appal. Math. 67 (2014) no. 8, 1336--1352.

\bibitem{HeSC01} W. Hebisch, L. Saloff-Coste, On the relation between elliptic and parabolic Harnack inequalities. Ann. Inst. Fourier (Grenoble) 51 (2001), no. 5, 1437--1481. 

\bibitem{K01} J. Kigami, Analysis on fractals, Cambridge University Press, Cambridge, 2001.

\bibitem{GS17} A. Grigor'yan and Y. Sun. On positive solutions of semi-linear elliptic inequlities on Riemannian manifolds, 2017

\bibitem{L83} E. H. Lieb. On the lowest eigenvalue of the Laplacian for the intersection of two domains. Invent. Math., 74(3) (1983) 441--448.

\bibitem{feyn} J. Lorinczi, F. Hiroshima and V. Betz, 
Feynman-Kac-type theorems and Gibbs measures on path space. 
With applications to rigorous quantum field theory. De Gruyter Studies in Mathematics, 34. Walter de Gruyter \& Co., Berlin, 2011. 
2nd ed. forthcoming, 2018

\bibitem{LS18} J. Lierl and S. Steinerberger, A local Faber-Krahn inequality and applications to Schr\"{o}dinger equations,  Comm. Partial Differential Equations 43 (2018), no. 1, 66--81.

\bibitem{MP99} E. Mitidieri and S. I. Pohozaev, Nonexistence of positive solutions for quasilinear elliptic problems on
RN , Proc. Steklov Inst. Math. 227 (1999) 186--216.

\bibitem{MP04} E. Mitidieri and S. I. Pohozaev. Towards a unified approach to nonexistence of solutions for a class of differential
inequalities. Milan J. Math., 72 (2004) 129--162

\bibitem{MP} Peter M\"{o}rters and Yuval Peres, Brownian motion. With an appendix by Oded Schramm and Wendelin Werner. Cambridge Series in Statistical and Probabilistic Mathematics, 30.
Cambridge University Press, Cambridge, 2010.

\bibitem{oneil} R. O'Neil, Convolution operators and $L(p,\,q)$ spaces. Duke Math. J. 30 (1963) 129--142.

\bibitem{RS18} M. Rachh and S. Steinerberger, On the location of maxima of solutions of Schr\"odinger's equation, to
appear in Comm. Pure Appl. Math

\bibitem{sim} B. Simon, Schr\"odinger semigroups. Bull. Amer. Math. Soc. (N.S.) 7 (1982), no. 3, 447--526. 

\bibitem{S95} M.A. Souto, A priori estimates and existence of positive solutions of nonlinear cooperative elliptic systems, Differential Integral Equations 8 (1995)
1245--1258.

\end{thebibliography}
\end{document}